\documentclass[11pt]{article}
\usepackage{amssymb,amsmath,latexsym}
\usepackage{enumerate}
\usepackage{graphicx}              % to include figures
\usepackage{amsfonts}              % for blackboard bold, etc
\usepackage{amsthm}      % better theorem environments
\usepackage[verbose]{wrapfig}          
\usepackage{url}
\usepackage{marvosym}
\usepackage{color}
\usepackage{todonotes}
\usepackage[backend=bibtex,style=alphabetic,
    doi=false,
    isbn=false,url=false]{biblatex}
\usepackage[hidelinks,pdfencoding=unicode,psdextra]{hyperref}
\usepackage{textcomp}
\usepackage{caption}
\usepackage{subcaption}
\usepackage{tikz,tikz-cd}
\usetikzlibrary{matrix, arrows}
\usepackage{mathtools}
\usepackage{bm}
\usepackage{csquotes}

\usepackage[OT2,OT1]{fontenc}
\newcommand\cyr
{
\renewcommand\rmdefault{wncyr}
\renewcommand\sfdefault{wncyss}
\renewcommand\encodingdefault{OT2}
\normalfont
\selectfont
}
\DeclareTextFontCommand{\textcyr}{\cyr}

\newtheorem{thm}{Theorem}[section]
\newtheorem*{thm*}{Theorem}
\newtheorem{lem}[thm]{Lemma}
\newtheorem{prop}[thm]{Proposition}
\newtheorem{cor}[thm]{Corollary}

\newtheorem*{main-thm}{Trivialization Theorem}
\newtheorem*{main-thm-reg}{Trivialization Theorem for $\bf\reg{Y}$}
\newtheorem*{conj-kpi1}{The $K(\pi,1)$ Conjecture}
\newtheorem*{conj-looi}{Looijenga's Conjecture}

\theoremstyle{definition}
\newtheorem{defn}[thm]{Definition}
\newtheorem{example}[thm]{Example}
\newtheorem{remark}[thm]{Remark}
\newtheorem{question}[thm]{Question}
\newtheorem*{remark*}{Remark}
\newtheorem*{meta-conj-1}{Meta-Conjecture I}
\newtheorem*{meta-conj-2}{Meta-Conjecture II}

\newlength{\currentparindent}

\newcommand{\op}[1]{\operatorname{#1}}

\newcommand{\CC}{\mathbb{C}}

\newcommand{\bpm}{\begin{pmatrix}}
\newcommand{\epm}{\end{pmatrix}}

\newcommand{\sym}{S}

\newcommand{\reg}[1]{#1^{\op{reg}}}

\newcommand{\orb}[1]{[#1]}

\usepackage{accents}

\makeatletter
\providecommand*{\twoheadrightarrowfill@}{%
  \arrowfill@\relbar\relbar\twoheadrightarrow
}
\providecommand*{\twoheadleftarrowfill@}{%
  \arrowfill@\twoheadleftarrow\relbar\relbar
}
\providecommand*{\xtwoheadrightarrow}[2][]{%
  \ext@arrow 0579\twoheadrightarrowfill@{#1}{#2}%
}
\providecommand*{\xtwoheadleftarrow}[2][]{%
  \ext@arrow 5097\twoheadleftarrowfill@{#1}{#2}%
}
\makeatother

\begin{document}

\title{Lyashko-Looijenga morphisms and primitive factorizations \\ of Coxeter elements}
\author{\Large Theo Douvropoulos\thanks{\href{mailto:tdouvropoulos@brandeis.edu}{tdouvropoulos@brandeis.edu}\quad This work was partially supported by the European Research Council, grant ERC-2016- STG 716083 ``CombiTop''.}}

\newcommand{\Address}{{% additional braces for segregating \footnotesize
  \bigskip
  \footnotesize
	
  Theo~Douvropoulos, \textsc{Brandeis University, Waltham, MA}\par\nopagebreak
  \textit{E-mail address:} \texttt{tdouvropoulos@brandeis.edu}	

}}

\date{}

\maketitle

\begin{abstract}
In the seminal work \cite{bes-kpi1}, Bessis gave a geometric interpretation of the noncrossing lattice $NC(W)$ associated to a well-generated complex reflection group $W$. A chief component of this was the {\it trivialization theorem}, a fundamental correspondence between families of chains of $NC(W)$ and the fibers of a finite quasi-homogeneous morphism, the $LL$ map.

We consider a variant of the $LL$ map, prescribed by the trivialization theorem, and apply it to the study of finer enumerative and structural properties of $NC(W)$. In particular, we extend work of Bessis and Ripoll and enumerate the so-called ``primitive factorizations" of Coxeter elements $c$. These are the length additive factorizations of the form $c=w\cdot t_1\cdots t_k$, where $w$ belongs to a given conjugacy class and the $t_i$'s are reflections.
\end{abstract}

\section{Introduction}

At the end of the 19th century, Hurwitz was one of the early disciples of Riemann surface theory. In \cite{hurwitz-riemannsche-flachen}, he translated the problem of counting the number of $n$-sheeted Riemann surfaces with given branched points into a question of enumerating a class of factorizations in the symmetric group $S_n$. He described an answer to the latter, a special case of which states: 

\begin{thm}\cite[\S~7]{hurwitz-riemannsche-flachen}\label{Thm: Hurwitz n^n-2}
In the symmetric group $S_n$, there are $n^{n-2}$ smallest length factorizations $t_1\cdots t_{n-1}=(12\cdots n)$ of the long cycle in transpositions $t_i$.
\end{thm}

This work by Hurwitz was a starting point for a plethora of results on the factorization enumeration of elements in the symmetric group and more general groups. Different versions of this question are amenable to a variety of techniques, ranging from bijective enumeration \cites{denes-overcounting}{goulden-yong}{schaeffer-vassilieva-bijective-Jackson}, to representation theory and the lemma of Frobenius \cites{stanley-factorization-of-permutations-into-n-cycles}{jackson-products-of-conjugacy-classes}{chapuy-stump-counting-factorizations}{reiner-lewis-stanton-singer-cycles}, to the Lagrange inversion formula for generating functions \cites{goulden-jackson-trees-and-cacti}{krattenthaler-Muller-decomposition-numbers}. 

Now, the symmetric group $S_n$ can be realized as a reflection group and many theorems regarding permutations in $S_n$ have generalizations that hold for some (wider) class of reflection groups. One might argue that this makes such statements exceptionally interesting. In particular, Hurwitz's Thm.~\ref{Thm: Hurwitz n^n-2} has a meaningful analog (see Section~\ref{Section:Complex reflection groups and their braid groups} for the terminology): 

\begin{thm}\label{Thm: W-Hurwitz number} 
For an irreducible, well-generated, complex reflection group $W$ of rank $n$, with set of reflections $\mathcal{R}$ and a Coxeter element $c\in W$, let $\op{Red}_{\mathcal{R}}(c)$ denote the set of smallest length factorizations $t_1\cdots t_n=c$ in reflections $t_i\in\mathcal{R}$. Then, if $h$ denotes the order of $c$, we have \begin{equation}|\op{Red}_{\mathcal{R}}(c)|=\dfrac{h^nn!}{|W|}.\label{Eq: W-Hurwitz number}\end{equation}
\end{thm}

This formula has a fascinating history. The calculation was first made for Weyl groups by Deligne (crediting Tits and Zagier) \cite{deligne-letter-to-looijenga} to prove a conjecture in singularity theory by Looijenga \cite[Conj.~(3.5)]{looij-complement-bifurc}. Deligne described a uniform recursion (later rediscovered by Reading \cite{reading-chains-in-noncrossing}) on the set of factorizations, leading to a recurrence for the cardinality which he then solved case-by-case. It was not observed that the answers might be expressed in a uniform way.

In fact, the first occurrence in the literature\footnote{However, it was Arnold \cite[Thm.~11]{arnold-critical-points-of-smooth-functions-ICM} who first used the expression on the RHS of \eqref{Eq: W-Hurwitz number} as a uniform formula for the quasi-homogeneous degree of the $LL$ morphism.} of the formula \eqref{Eq: W-Hurwitz number} appears in \cite[Prop.~9]{chapoton-enumerative-properties}, where reduced reflection factorizations of $c$ are interpreted as maximal chains in the noncrossing lattice $NC(W)$. Its derivation there relies on a uniform formula for the zeta polynomial of $NC(W)$, proven again case-by-case.

For arbitrary (irreducible, well-generated) complex reflection groups $W$, formula \eqref{Eq: W-Hurwitz number} is proven by Bessis \cite[Prop.~7.6]{bes-kpi1}, with combinatorial arguments for the infinite family and computer calculations 
for the exceptional cases. The lemma of Frobenius allows us to generalize the enumeration to arbitrary length factorizations \cites{chapuy-stump-counting-factorizations}{reiner-delmas-hameister-refined-count}{theo_chapuy_jucys}, but still fails, on its own, to explain the uniform formula.

It has only been recently, and only for real groups $W$, that case-free derivations of this formula were produced. The case of Weyl groups was done in \cite{jean-michel-case-free}, while \cite{theo-malle} covered all real groups. Both of these works rely on representation-theoretic techniques and in fact prove the stronger formula of \cite{chapuy-stump-counting-factorizations}. An elementary, case-free, Coxeter-theoretic proof of Thm.~\ref{Thm: W-Hurwitz number} first appeared in \cite{theo-Laplacian}.

\subsubsection*{A geometric interpretation for $\op{Red}_{\mathcal{R}}(c)$}

For our work, a different {\it geometric} exegesis of the set of reduced reflection factorizations $\op{Red}_{\mathcal{R}}(c)$ is most significant. They are related to the quasi-homogeneous Lyashko-Looijenga ($LL$) morphism, a map originating in singularity theory, but which Bessis further defined for all irreducible, well-generated, complex reflection groups $W$. 

The $LL$ map essentially describes the discriminant hypersurface $\mathcal{H}$ of $W$ as a branched covering along the direction of the highest degree invariant $f_n$ (see Defn.~\ref{Defn: LL map}). Bessis' trivialization theorem explains the relation with block factorizations; we postpone the complete statement until Section~\ref{Section: The trivialization theorem} and provide a lighter version for now:

\begin{main-thm}
The elements in a generic fiber of the $LL$ map are in a natural $1$-$1$ correspondence with the set $\op{Red}_{\mathcal{R}}(c)$ of reduced reflection factorizations of a Coxeter element \nolinebreak$c$.
\end{main-thm}

Here, it is important to warn the reader that the current proof of the trivialization theorem relies on Thm.~\ref{Thm: W-Hurwitz number} (see Prop.~\ref{Prop: Numerological coincidence}). Bessis constructs a labeling map that assigns factorizations to elements in the fiber of the $LL$ map, but neither surjectivity nor injectivity can be proven {\it a priori}. 

The $LL$ map and the trivialization theorem are fundamental in Bessis' proof of the $K(\pi,1)$ conjecture, namely that the complement $V^{\op{reg}}:=V\setminus\bigcup H$ of the reflection arrangement of $W$ has a contractible universal covering space. Bessis uses the $LL$ map to produce the dual braid presentation of the generalized braid group $B(W):=\pi_1(W\setminus V^{\op{reg}})$, and the combinatorics of $\op{Red}_{\mathcal{R}}(c)$ (or really of $NC(W)$) to construct the universal cover (which is identical for $W\setminus V^{\op{reg}}$ and $V^{\op{reg}}$) and to show that it is contractible.

The $K(\pi,1)$ conjecture for reflection arrangements was one of the most significant questions in the theory of hyperplane arrangements. Its importance stems from the fact that it guarantees a simple calculation of the (group-theoretic) cohomologies of the pure braid groups $P(W)$ (this was actually the context when Brieskorn first stated the conjecture for real $W$; see \cite[\S~2]{brieskorn-sur-les-groupes-de-tresses}). 

This line of research was inspired by work of Arnold on algebraic functions, their discriminants, and the resulting interactions with the ordinary braid group $B_n$. For a detailed presentation of this story, and the connection with Arnold's pioneering ideas for the {\it algebraic} version of Hilbert's thirteenth problem, see \cite[\S~1.1]{theo-thesis}.

\subsubsection*{Enumeration via degree counting}

In the symmetric group case (see also Section~\ref{sec:LL_on_polys}), the original definition of the $LL$ map describes it as the morphism that sends a polynomial $p(z)\in \CC[z]$ to its multiset of critical values. Under this interpretation, the bijective correspondence of the trivialization theorem above is guaranteed by Riemann's existence theorem \cite[Thm.~1.8.14]{LZ-graphs-on-surfaces}. Even though this is implicit in Looijenga's proof of the type-$A$ case \cite[(3.6)]{looij-complement-bifurc}, it was 20 years later and it was Arnold \cite{arnold-first-LL-application} who first realized that the $LL$ map can thus be used to {\it produce} enumerative results.\footnote{In fact Arnold was known to state \cite{arnold-first-LL-application} that {\it ``The simplest way to prove this theorem of Cayley} [the tree enumeration, equivalent to Thm.~\ref{Thm: Hurwitz n^n-2}] {\it is perhaps to count the multiplicity of the quasi-homogeneous Lyashko-Looijenga mapping"}.}

The number of elements of a generic fiber is a classical invariant of finite morphisms, called the {\it degree} of the map; it is not always easy to compute. Nevertheless, when the morphism is quasi-homogeneous, as is the case with the $LL$ map, Bezout's theorem \cite[Thm.~5.1.5]{LZ-graphs-on-surfaces} provides a very simple formula for the degree. Specifically, if the map $F$ is given as $$ \CC^n\ni{\bm u}:=(u_1,\cdots,u_n)\xrightarrow{\ \ F \ \ } \big(f_1({\bm u}),\cdots,f_n({\bm u})\big)\in\CC^n,$$ where the $f_i$'s are quasi-homogeneous polynomials in the $u_i$'s, then its degree equals \begin{equation}\op{deg}(F)=\dfrac{\prod_{i=1}^n\op{deg}(f_i)}{\prod_{i=1}^n\op{deg}(u_i)}.\label{Eq: Bezout's theorem}\end{equation}  

After Arnold, there was a wealth of work \cites{goryunov-lando-enumeration-meromorphic}{elsv}{LZ-mult-of-LL}{baines-LL-even-and-odd}{LZ-counting-I} where variants of the $LL$ map were considered in order to enumerate combinatorial objects usually associated with factorizations in the symmetric group $S_n$. The degree calculations are not always easy (especially so in \cite{elsv}), but a common idea is to {\it lift} the $LL$ map to a suitable domain (often some affine space $\CC^N$), where it becomes quasi-homogeneous.   

\subsubsection*{Main Theorem}

The purpose of this paper, is to advertise and apply such geometric techniques, in the context of Bessis' $LL$ map and (irreducible, well-generated) complex reflection groups. In particular, the parabolic stratification of the discriminant hypersurface $\mathcal{H}$ of $W$ (Section~\ref{Section: The parabolic stratification of H}) allows for a local study of the $LL$ map that leads to the enumeration of finer sets of factorizations of the Coxeter element $c\in W$. Any length additive factorization of $c$ is called a {\it block factorization} (see Section~\ref{Section:Complex reflection groups and their braid groups} for more details), while block factorizations where only the first factor is not a reflection are called {\it primitive}.

\begin{thm*}[Section~\ref{Section: Primitive factorizations of a Cox elt}]
Let $W$ be a well-generated complex reflection group, acting irreducibly on the space $V$, and let $Z$ be a flat of the reflection arrangement of $W$. Then, the number $\op{Fact}_W(Z)$ of primitive factorizations of a Coxeter element $c\in W$ of the form $w\cdot t_1\cdots t_k=c$ where $V^w$ belongs to the $W$-orbit of $Z$, the $t_i$'s are reflections, and $k=\op{dim}(Z)$, is given by:$$\op{Fact}_W(Z)=\dfrac{h^k\cdot k!}{[N_W(Z):W_Z]}.$$ Here $h$ is the Coxeter number of $W$, and $N_W(Z)$ and $W_Z$ the setwise and pointwise stabilizers of $Z$ respectively.
\end{thm*}  

\begin{example}
Let $W$ be the symmetric group $S_4=A_3$, seen via its reflection representation as a subgroup of $\operatorname{GL}(\mathbb{C}^3)$, and denote its (reflection) hyperplanes by $H_{ij}$ ($1\leq i < j \leq 4$). Consider the $1$-dimensional flat $Z=H_{13}\cap H_{24}$ and the Coxeter element $c=(1234)\in S_4$ (written in cycle notation). 

\noindent The $S_4$-orbit of $Z$ consists of the three flats $V^{w_i}$ for  elements $(w_i)_{i=1}^3=\Big((12)(34),(13)(24), (14)(23)\Big)$. Among these, $w_1$ and $w_3$ are noncrossing, and they give us the factorizations 
\[
(12)(34)\cdot (24)=(1234) \quad\quad\text{and}\quad\quad (14)(23)\cdot (13)=(1234),
\]
while $w_2$ is crossing and thus there is no reflection $t$ with $(13)(24)\cdot t=(1234)$. That is, $\op{Fact}_{S_4}(Z)=2$.

\noindent Now, we can also calculate the right hand side of the equation in the theorem:
\[
N_{S_4}(Z)=\Big\{e,(13),(24),(13)(24),(1234),(14)(23),(12)(34),(1432)\Big\},
\]
while 
\[
(S_4)_Z=\{e,(13),(24),(13)(24)\},
\]
and the formula gives $\frac{h^kk!}{[N_W(Z):W_X]}=\frac{4^11!}{8/4}=2$ confirming the theorem for this case.
\end{example}

Our formula can easily be seen to generalize the formulas of Bessis (Thm.~\ref{Thm: W-Hurwitz number}) and Hurwitz (Thm.~\ref{Thm: Hurwitz n^n-2}), by setting $Z=V$ (see also Rem.~\ref{Remark:Eval at H and H'}) and by further setting $W=\sym_n$ respectively. It should also be considered as a further generalization of \cite[Thm.~4.1]{rip-submax}, although we use a different approach. In particular, Ripoll's formula only allows flats $Z$ of codimension $2$, and is in a sense less explicit than ours. The results of \cite{krattenthaler-Muller-decomposition-numbers} on the other hand, recover the previous theorem for real $W$, but they are based on case-by-case considerations for the infinite families and computer calculations for exceptional groups, and the uniform formula is not observed. Moreover, our main Thm.~\ref{Thm: Primitive facto enume} also leads to an explicit ramification formula for the $LL$ map which gives a geometric proof (Thm.~\ref{thm:parab_rec_cox}) for a parabolic recursion relating the Coxeter number $h$ of $W$ with the Coxeter numbers of its parabolic subgroups.

To prove our main Thm.~\ref{Thm: Primitive facto enume}, we first lift the $LL$ morphism to a map $\widehat{LL}$ with domain $Z$ and target a {\it decorated} configuration space. Then, we compare the degree of the new map $\widehat{LL}$ with the number of primitive factorizations via the trivialization theorem. The index $[N_W(Z):W_Z]$ appears naturally as an overcounting factor. 

As we mentioned earlier, the trivialization theorem relies on the enumeration result of Thm.~\ref{Thm: W-Hurwitz number}. Our results are uniform extensions of it, but still depend on it non-trivially. In particular, the proof of  Thm.~\ref{Thm: Primitive facto enume} is case-free for the same class of groups that the trivialization theorem holds uniformly (currently all real reflection groups).

\subsubsection*{Summary}

In Section~\ref{Section:Complex reflection groups and their braid groups}, we recall necessary background for complex reflection groups and braid groups. In the following Sections~\ref{Section: Geometric facto's} and~\ref{Section: The LL map and the trivialization theorem}, we review part of Bessis' theory with particular attention to the geometric properties of the $LL$ map upon which our work relies. Moreover, in Section~\ref{sec: geom of LL} we fill in some non-trivial gaps of \cite{bes-kpi1} and replace a faulty geometric argument for the proof of the finiteness property of the $LL$ map. We hope this will be useful to the community as the finiteness property is the main geometric reason for the nice behavior of the $LL$ map. For instance, in his survey \cite{vassiliev-a-few-problems-on-monodromy-and-discriminants} Vassiliev notes that a main obstruction to applying the $LL$ technique is that the $LL$ map might not be proper (this happens for many non-simple singularities). We add to this a Section~\ref{sec:LL_on_polys}, where we describe the relationship between Bessis' $LL$ map and the  classical $\mathcal{LL}$ map on polynomials.

In Section~\ref{Section: Primitive factorizations of a Cox elt} we introduce the morphism $\widehat{LL}$, study its geometry, and prove our main theorem. In Section~\ref{Section: A geometric interpretation of Kreweras numbers}, we speculate on how our lifted $\widehat{LL}$ map might be used towards a uniform, equivariant enumeration of noncrossing partitions. Finally, in the last section~\ref{Section: The shadow stratification}, we define and study the {\it shadow stratification} associated to the group $W$. This is in a sense the canonical geometric object related with the enumeration (Prop.~\ref{Prop: local degree of LL and Fact}) and structural properties (Thm.~\ref{Thm: Hurwitz action and path connectedness}) of block factorizations of Coxeter elements. We hope that the framework we build there might be used towards a possible {\it geometric} generalization of the Goulden-Jackson formula, uniformly given for reflection groups (see Section~\ref{Section: The complete answer is known in type A}).

\section{Complex reflection groups and their braid groups}
\label{Section:Complex reflection groups and their braid groups}

A {\color{blue} complex reflection group} $W$ is a finite subgroup of $\op{GL}(V)$, for some space $V\cong \CC^n$, that is generated by unitary reflections (also called {\it pseudo-reflections} later on). These are $\CC$-linear maps $t$ that fix a hyperplane in $V$, i.e. for which $\op{codim}(V^t)=1$. We denote by $\mathcal{R}$ the set of (all) unitary reflections of $W$ and by $\mathcal{A}$ (or $\mathcal{A}_W$) the set of their fixed hyperplanes. The group $W$ is called {\color{blue} irreducible} if it leaves no nontrivial subspace of $V$ invariant. Shephard and Todd \cite{shephard-and-todd} classified such groups $W$ into one infinite, 3-parameter family $G(d,r,n)$, and 34 exceptional cases (indexed $G_4$ to $G_{37}$). Most of the material in this section can be found in the classical references \cites{broue-book-braid-groups}{kane-book-reflection-groups}{lehrer-taylor-unitary-reflection-groups}.

Complex reflection groups act on the polynomial algebra $\CC[V]:=\op{Sym}(V^*)$ of the ambient space by precomposition; that is, we define $(w\star f)(v):=f(w^{-1}\cdot v)$. The Shephard-Todd-Chevalley theorem \cite{chevalley-improves-Shep-and-Todd} states that, under the action of a complex reflection group $W$, the invariant algebra $\CC[V]^W:=\{f\in \CC[V] \ |\ w\star f=f\  \text{for all}\ w\in W\}$ is itself a polynomial algebra and of the same Krull dimension $n$. We write $f_i$ for its generators (that is, $\CC[V]^W=\CC[f_1,\cdots,f_n]$) and we call them the {\color{blue} fundamental invariants} of $W$. We choose them to be homogeneous and index them by increasing degree order (i.e. $\op{deg}(f_i)\leq\op{deg}(f_{i+1})$). Their degrees $d_i:=\op{deg}(f_i)$ are then independent of our choice of $f_i$'s; we call them the {\color{blue} degrees} of $W$.  

Since they are finite and act linearly, complex reflection groups already have a very simple {\it Geometric Invariant Theory} (GIT). The Shephard-Todd-Chevalley theorem implies that, in fact, they have the best possible; the orbit space $W\setminus V$ is an affine complex space ($W\setminus V\cong \CC^n$). Moreover the quotient map $\rho: V\rightarrow W\setminus V$ is given {\it explicitly} via the fundamental invariants:

\begin{equation}\label{Eq: quotient map rho}
\CC^n\cong V\ni {\bm z}:=(z_1,\cdots,z_n)\xrightarrow{\rho}{\bm f(\bm z)}:=(f_1({\bm z}),\cdots,f_n({\bm z})\big)\in W\setminus V\cong\CC^n
\end{equation} 

An irreducible complex reflection group $W$ acting on an $n$-dimensional space, can always be generated by either $n$ or $n+1$ reflections (fewer reflections are never sufficient). In what follows we will restrict ourselves to groups in the first category, which are called {\color{blue} well-generated}. Of their various properties and equivalent characterizations \cites[Prop.~4.2]{bes-zariski}[Thm.~2.4]{bes-kpi1}, key for us is that they have good analogs of Coxeter elements.

Recall first the {\it generalized Coxeter number} $h$ of $W$, as defined by Gordon and Griffeth \cite{gordon-griffeth-catalan}. It is given by $\displaystyle h:=\frac{N^*+N}{n}$, where $N^*:=|\mathcal{R}|$ and $N:=|\mathcal{A}|$ are, respectively, the numbers of pseudo-reflections and of reflecting hyperplanes of $W$, and $n$ is its rank. For a well-generated group $W$, $h$ is equal to the highest degree $d_n$. Then we have:

\begin{defn}\cite[Defn.~7.1]{bes-kpi1}\label{Defn:Cox elt}
We set $\zeta=e^{2\pi i/h}$ and define a {\color{blue} Coxeter element} $c$ of $W$ to be a $\zeta$-regular element in the sense of Springer \cite{springer-regular-elements}. That is, $c$ has a $\zeta$-eigenvector $v$ that lies in no reflection hyperplane. In well-generated groups $W$, Coxeter elements exist\footnote{In fact, this is yet another characterization of well-generated groups, as one can easily see from \cite[Prop.~4.2]{bes-zariski} and known properties of regular numbers.} and they form a single conjugacy class (but see \cite{reiner-ripoll-stump-non-conjugate-coxeter-elements} for a generalization).
\end{defn}

\subsubsection*{Block factorizations of $c$ and the Hurwitz action on them}

The {\it \color{blue} reflection length} $l_{\mathcal{R}}(w)$ of an element $w\in W$, is the smallest number $s$ of (pseudo)-reflections $t_i$ needed to factor $w=t_1\cdots t_s$ (this shouldn't be confused with the \emph{Coxeter length} in real reflection groups where only \emph{simple} reflections may be used; sometimes authors use the term \emph{absolute reflection length} instead of ``reflection length" to emphasize the distinction). This length function determines a partial order $\leq_{\mathcal{R}}$ (the {\it absolute order}) on the elements of $W$: $$u\leq_{\mathcal{R}} v \iff l_{\mathcal{R}}(u)+l_{\mathcal{R}}(u^{-1}v)=l_{\mathcal{R}}(v).$$ 

The {\it \color{blue} noncrossing lattice $NC(W)$} is the interval $[1,c]_{\leq_{\mathcal{R}}}$, in the absolute order $\leq_{\mathcal{R}}$, between the identity $1$ and an arbitrary Coxeter element $c$. Since conjugation respects the set of reflections, the various choices of $c$ give isomorphic lattices and for this reason we use the notation $NC(W)$ instead of a more cumbersome $NC(W,c)$. An element $c_i$ will be further called {\it noncrossing with respect to $c$}, for some Coxeter element $c$, if $c_i\leq_{\mathcal{R}} c$. 

We call an expression $c=w_1\cdot w_2\cdots w_k$ a {\it (reduced)} {\it \color{blue} block factorization} of $c$, if it is length additive; that is, if \begin{equation}l_{\mathcal{R}}(c)=l_{\mathcal{R}}(w_1)+\cdots l_{\mathcal{R}}(w_k).\label{EQ: block factorizations}\end{equation} If all the $w_i$'s are moreover pseudo-reflections, we call it a {\it reduced reflection factorization} of $c$. The set of all reduced reflection factorizations of $c$ is denoted by $\op{Red}_{\mathcal{R}}(c)$ and is in bijection with the set of maximal chains of $NC(W)$.

\begin{defn}\label{Defn: The (algebraic) Hurwitz action}
For {\it any} group $G$, there is a natural action of the braid group on $k$ strands $B_k$ on the set of $k$-tuples of elements of $G$. The generator $s_i$ acts via:
$$ s_i*(g_1,\cdots,g_{i-1},\ \ g_i,g_{i+1},\ \ g_{i+2},\cdots,g_k)=(g_1,\cdots,g_{i-1},\ \ g_{i+1},g^{-1}_{i+1}g_ig_{i+1},\ \ g_{i+2},\cdots,g_k).$$
We call this the (right) {\it \color{blue} Hurwitz action} of $B_k$ on $G^k$. It respects the product of the elements $g_i$ and is therefore well defined on the set of block factorizations of $c$.
\end{defn}

\subsection{The braid group \texorpdfstring{$B(W)$}{B(W)} and the discriminant hypersurface \texorpdfstring{$\mathcal{H}$}{H}}

In 1925 Artin \cite{artin-braid-groups-1925} introduced the braid group on $n$ strands $B_n$ and gave a constructive proof of its presentation: $$B_n:=\langle s_1,\cdots,s_{n-1} |\ s_{i+1}s_is_{i+1}=s_is_{i+1}s_i,\ s_is_j=s_js_i,\ j\neq i\pm 1\ \rangle.$$
He noticed that one can obtain the symmetric group $S_n$ by imposing the extra conditions $s_i^2=1$ \cite[Satz~3]{artin-braid-groups-1925}. There are natural generalizations of $B_n$ that are related to the other reflection groups in a similar fashion. The following definition, first by Brieskorn \cite{brieskorn-braid-groups-have-artin-like-presentations}, introduces them in the most suitable way for our geometric study of Coxeter elements and 
their factorizations:

\begin{defn}
Let $W\leq \op{GL}(V)$ be a complex reflection group and $V^{\op{reg}}$ the set of points in $V$ that have a trivial stabilizer under the $W$-action. We define the {\color{blue} braid group $B(W)$} to be the fundamental group of the space of regular orbits of $W$: $$B(W):=\pi_1(W\setminus V^{\op{reg}}).$$
\end{defn}

It is a theorem of Steinberg that the pointwise $W$-stabilizer of any $x\in V$ is generated by those reflections whose hyperplanes contain $x$. In particular, $W$ acts freely precisely on the complement of the reflection arrangement $\mathcal{A}_W$ of $W$ (that is, $V^{\op{reg}}=V\setminus \bigcup H$). We call the fundamental group $\pi_1(V^{\op{reg}})$ the {\color{blue} pure braid group} $P(W)$. The following short exact sequence, which is an immediate corollary of covering space theory \cite[Prop.~1.40]{hatcher-book-alg-top}, defines a surjection $\bm\pi:B(W)\twoheadrightarrow W$, analogous to the one between $B_n$ and $S_n$:
\begin{equation}\label{EQ:1PBW1}
1\rightarrow\underset{\overset{\rotatebox{90}{\,:=}}{\displaystyle P(W)}}{\pi_1(\reg{V})}
\xhookrightarrow{\ \displaystyle \rho_*\ }\underset{\overset{\rotatebox{90}{\,:=}}{\displaystyle B(W)}}{\pi_1(W\backslash\reg{V})}
\xtwoheadrightarrow{\ \displaystyle \bm\pi} W^{\op{op}}\rightarrow 1
\end{equation}

After a choice of a basepoint $v\in V^{\op{reg}}$, a loop $b\in B(W)$ lifts to a path that runs from $v$ to some final point which we denote $b_*(v)$ (this is called the {\it Galois action} of $b$ on $v$). Since now $v$ and $b_*(v)$ are in the same (free) $W$-orbit, there is a unique element $w\in W$ such that $w\cdot v=b_*(v)$. We define $\bm\pi(b):=w$ and obtain the surjection $\bm\pi$ which, after the choice of the basepoint $v$, we consider fixed.

The combinatorics of a real reflection group is to a great extent governed by its reflection arrangement $\mathcal{A}_W=\bigcup H$ which decomposes the ambient space into chambers. In the complex case, where a chamber decomposition does not exist, it is often more effective to look at the quotient $W\setminus \mathcal{A}_W$. The latter is in fact a variety (another consequence of GIT) which we call the {\color{blue} discriminant hypersurface} $\mathcal{H}$ of $W$. It is the zero set of a single polynomial in the fundamental invariants $f_i$'s, which we denote $\Delta(W,{\bm f})$, and which we also call the {\it discriminant} of $W$. Given a choice of linear forms $\alpha_H$ that define the hyperplanes $H\in\mathcal{A}_W$, the discriminant of $W$ is given as follows:
\begin{equation}
\Delta(W,{\bm f})=\prod_{H\in \mathcal{A}_W}\alpha^{}_H(\bm z)^{e^{}_H},\label{eq:discr_prod_aH_eH}
\end{equation}
where $e_H$ is the order of the cyclic group $W_H$ (and $e_H=2$ always for real $W$). The right hand side of \eqref{eq:discr_prod_aH_eH} is $W$-invariant, so that the left hand side is indeed a polynomial in the $f_i$. Now, its particular structure as a polynomial in the $f_i$ defines a criterion for being well-generated, which from now on will serve as our main assumption on $W$:

\begin{prop}[{for real $W$ see Sec.~3 of \cite{saito-on-a-linear-structure}, for the general case see Thm.~2.4 of \cite{bes-kpi1}}]\label{Prop: Discriminant is monic}\ \newline
Let $W$ be an irreducible complex reflection group. Then $W$ is well-generated if and only if for any system of basic invariants ${\bm f}$, we have that $\Delta(W,{\bm f})$, viewed as a polynomial in the highest degree invariant $f_n$, is {\bf monic and of degree ${\bm n}$}. In that case, we can further reduce it to:
\begin{equation}\label{Eq:Discriminant monic}
\Delta(W,{\bm f})=f_n^n+\alpha_2f_n^{n-2}+\cdots +\alpha_n,
\end{equation} where $\alpha_i\in\CC[f_1,\cdots,f_{n-1}]$, are quasi-homogeneous polynomials of weighted degree $\op{deg}(\alpha_i)=hi$. 
\end{prop}

\begin{remark}
The significance of the short exact sequence~\eqref{EQ:1PBW1} is  that it realizes the complex reflection group $W$ as the group of deck transformations of a particularly nice covering map $\rho:V^{\op{reg}}\rightarrow W\setminus V^{\op{reg}}$. Indeed, the covering $\rho$ is given by explicitly known polynomials (eq.~\ref{Eq: quotient map rho}) and its base space and cover space are complements of easily computable varieties.
\end{remark}

\subsection{The parabolic stratification of \texorpdfstring{$\mathcal{H}$}{H}}
\label{Section: The parabolic stratification of H}

We denote by $\mathcal{L}_W$ the intersection lattice of the reflection hyperplanes of $W$. Its elements $Z\in\mathcal{L}_W$ are called flats and they give a stratification of the ambient space $V$. We will use the symbol $\reg{Z}$ to indicate the {\it regular} part of a flat $Z$, that is, $\reg{Z}:=Z\setminus \bigcup_{H\not\supset Z}H$. The pointwise stabilizer of $Z$ is denoted by $W_Z$ and is itself a reflection group (after Steinberg's theorem). We call such $W_Z$ {\it \color{blue} parabolic subgroups} and we call their Coxeter elements {\it \color{blue} parabolic Coxeter elements}.

The quotient map $\rho:V\rightarrow W\backslash V$ induces the orbit stratification of the discriminant hypersurface $\mathcal{H}$, the strata of which are the $W$-orbits of the flats $Z$. We denote them by $\orb{Z}\in W\backslash \mathcal{L}_W$. As before, we will use the symbol $\orb{\reg{Z}}$ for the {\it regular} part of $\orb{Z}$.

The local topology of the reflection arrangement is very well understood (see \cites[Prop.~2.29]{broue-malle-rouquier-braid-hecke}[Prop.~20]{theo-thesis}). Around a point $p\in\reg{Z}$, the reflection arrangement $\mathcal{A}_W$ looks like the direct product of the flat $Z$ and the arrangement $\mathcal{A}_{W_Z}$ of the parabolic subgroup $W_Z$. This is reflected in the hypersurface $\mathcal{H}$ which, near a point $\orb{p}\in\orb{\reg{Z}}$, looks like the product of $Z$ and the discriminant hypersurface $\mathcal{H}(W_Z)$. This local behavior induces an embedding of the corresponding braid groups $B(W_Z)\hookrightarrow B(W)$, which is well defined up to conjugation.

Another consequence of this is the following lemma which relates the local {\it geometry} of the discriminant $\mathcal{H}$ with the combinatorics of the hyperplane arrangement. The multiplicity of a scheme at a point is a numerical invariant that records in some sense how much the scheme fails to be smooth at that point. For a hypersurface in $\CC^n$, given as the zero set of a polynomial $F({\bm x})$, the multiplicity at the origin is the smallest degree of the monomials of $F$ (see \cite[Geometric Interludes, No.~1]{theo-thesis} for more details). For instance, equation \eqref{Eq:Discriminant monic} implies that in an irreducible, well-generated, complex reflection group, the multiplicity of $\mathcal{H}$ at the origin ${\bm 0}$ is equal to $n$, the degree of the monomial $f_n^n$. This generalizes to the following lemma (see \cite[Lemma~5.4]{bes-kpi1} and for a longer exposition \cite[Lemma~34]{theo-thesis}).

\begin{lem}\label{Lemma: mult=codim}
Let $W$ be an irreducible, well-generated complex reflection group, $[p]$ a point in the discriminant hypersurface $\mathcal{H}$, and let $Z$ be a flat such that $[p]\in\orb{\reg{Z}}$. Then,$$\op{mult}_{[p]}(\mathcal{H})=\op{codim}(Z).$$
\end{lem}

\section{Geometric factorizations of the Coxeter element}
\label{Section: Geometric facto's}

We briefly review in this section Bessis' geometric-topological construction of factorizations of the Coxeter element. The reader may also consult \cites[Sec.~3.2]{rip-submax}[Sec.~6]{bes-kpi1}[Sec.~4]{theo-thesis}.

As we mentioned earlier, GIT implies a realization of the quotient $W\setminus V$ as the affine $n$-dimensional space $\CC^n$ (eq.~\ref{Eq: quotient map rho}). The derivation along the last coordinate, which corresponds to the highest degree invariant $f_n$, is of particular importance for well-generated groups (compare with Saito's primitive form \cite[\S~1.6]{saito-polyhedra-dual-a-precis}). For our purposes, we need to construct loops in $B(W)$ that surround the discriminant $\mathcal{H}$ only along the direction of $f_n$.

Towards that end, we introduce the {\color{blue} base space} $Y=\op{Spec}\CC[f_1,\cdots,f_{n-1}]$, so that $W\setminus V\cong Y\times \CC$, with coordinates $(y,x)$ or sometimes $(y,f_n)$. The slice $L_0:={\bm 0}\times\CC$, given by $f_i=0$, $i=1,\cdots, n-1$ and $f_n$ arbitrary, intersects the discriminant $\mathcal{H}$ solely at the origin $({\bm 0},0)$ (by eq.~\ref{Eq:Discriminant monic}). The point $({\bm 0},1)$ lies therefore in $W\setminus V^{\op{reg}}$; we pick some $v$ in its preimage in $V^{\op{reg}}$ and set it, now and for all, as the basepoint of our covering map $\rho$ (eq.~\ref{EQ:1PBW1}). 

Consider now (see Fig.~\ref{Fig:label map}) the loop in $L_0$ given by $f_n(t)=e^{2\pi i t},\ t\in [0,1]$. It defines an element $\delta\in B(W)$, whose Galois action sends $v$ to some point $v'=\delta_*(v)$. It is easy to see (by the homogeneity of the $f_i$'s) that the lift of $\delta$ is the path $e^{2\pi i t/h}\cdot v$, $t\in [0,1]$; it traces a rotation of $v$ by $2\pi/h$ radians. That is, $v$ is a $e^{2\pi i/h}$-eigenvector of the element $\bm\pi(\delta)\in W$ (as in eq.~\ref{EQ:1PBW1}). We conclude that:
\begin{prop}\cite[Lem.~6.13]{bes-kpi1}
The element $c:=\bm\pi(\delta)$ is a Coxeter element according to Defn.~\ref{Defn:Cox elt}. Different choices of $v$ will give the whole conjugacy class of Coxeter elements.
\end{prop}

\subsection{The labeling map \texorpdfstring{$\op{rlbl}$}{rlbl}}
\label{Section: The labeling map rlbl}

According to the previous discussion, the special slice $L_0$, which does not depend on the choice of fundamental invariants ${\bm f}$ (see \cite[Rem.~25]{theo-thesis}), gives rise to a {\it geometric construction} of the Coxeter element. We now consider an arbitrary slice $L_y$ to produce {\it factorizations} of $c$. The intersection $L_y\bigcap \mathcal{H}$ will now be comprised of $n$ points (counted with multiplicity), namely the solutions of the equation
\begin{equation}\label{Eq: Discrim. eqtn}
\big(\Delta(W,{\bm f});(y,t)\big):= t^n+\alpha_2(y)t^{n-2}+\cdots+\alpha_n(y)=0.
\end{equation}
Here $y\in Y$ is fixed and $t$ is the unknown, while the $\alpha_i$'s are as in \eqref{Eq:Discriminant monic}. We write $(y,x_i)$ for the solutions.

Bessis \cite[\S~6 and Defn.~7.14]{bes-kpi1} describes a way of drawing loops around these points $(y,x_i)$ and shows that, via the fixed surjection $\bm\pi$, they map to factorizations of the Coxeter element. This process of assigning factorizations to points $y$ of the base space will be called a {\color{blue}labeling map}. The construction requires an ordering of the complex numbers $x_i$. For us, this will be the {\it complex-lexicographic} one; that is, we order the points $x_i\in \CC$ by increasing real part first, and break ties by increasing imaginary part.

As all our loops should eventually be based at $({\bm 0},1)$, we start by picking a path $\theta$ in $Y$ that connects ${\bm 0}$ to $y$. We lift it to a path $\beta_{\theta}$ in $Y\times\CC$, starting at $({\bm 0},1)$, which always stays ``above'' (i.e. has bigger imaginary part than) all the points in the intersections $L_{y'}\bigcap\mathcal{H}$ (see Fig.~\ref{Fig:label map}). We call the endpoint of this path $(y,x_{\infty})$ to indicate that it lies in the slice $L_y$ and above all points $(y,x_i)$.

\begin{wrapfigure}{r}{0.5\textwidth}%\vspace{-0.6cm}
  \begin{center}\vspace{-1.1cm}
	\includegraphics[trim={12cm 2.5cm 0 1.5cm},clip,width=0.64\textwidth]{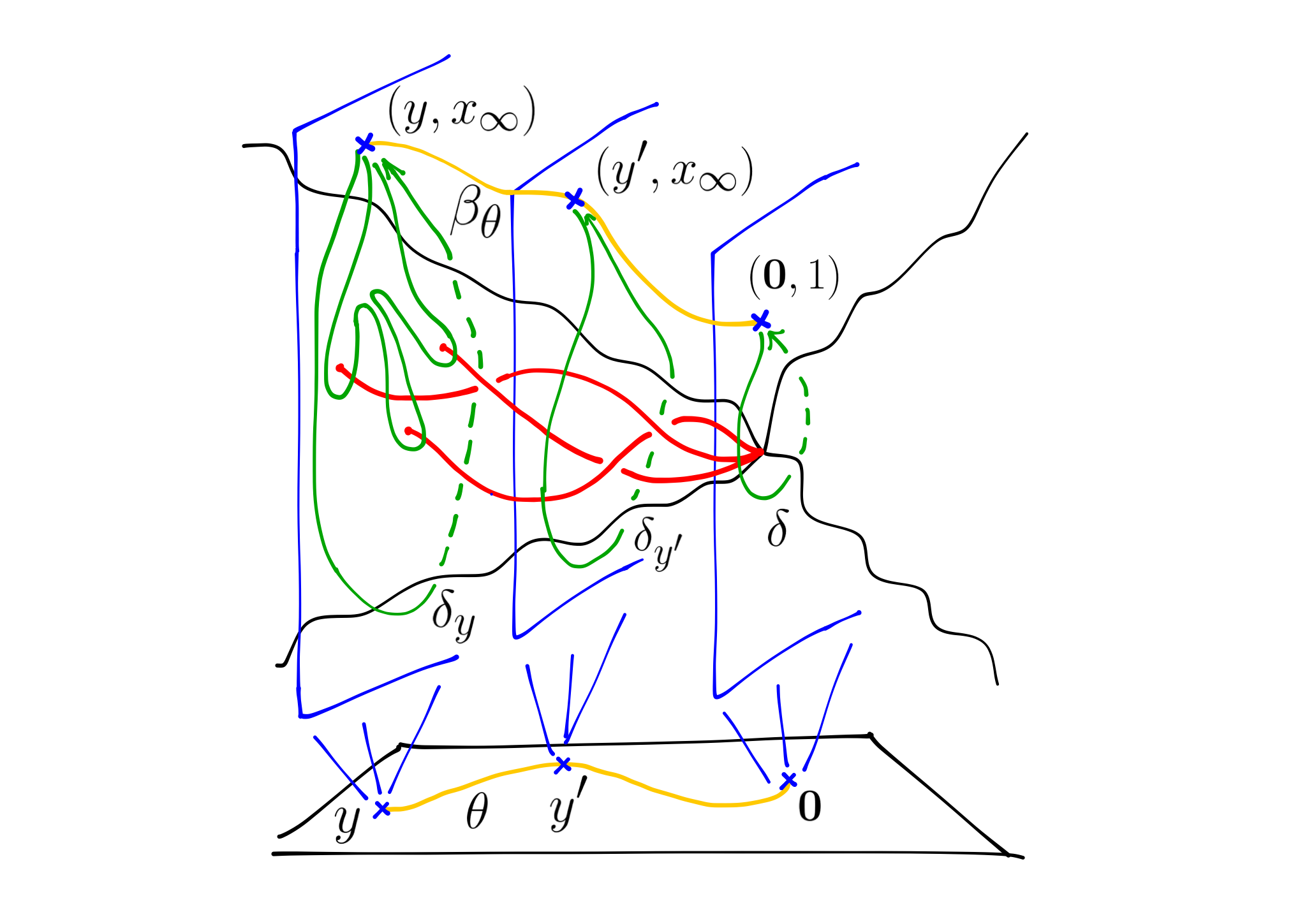}
  \end{center}\vspace{-0.75cm}
\caption{Geometric factorizations via the slice $L_y$.}\label{Fig:label map}\vspace{-0.3cm}
\end{wrapfigure}

From $x_{\infty}$ we now construct paths $\beta_i$ in $L_y$ down to the points $x_i$ such that they never cross each other (or themselves), and their order as they leave $x_{\infty}$ is given by their indices (i.e. $\beta_1$ is the leftmost one).

Given this information, we can now easily construct elements $b_{(y,x_i)}$ of $B(W)$: First, we follow the path $\beta_{\theta}$ from the basepoint $({\bm 0},1)$ to $(y,x_{\infty})$, then we go down $\beta_i$ but before we reach its end, we trace a small counterclockwise circle around $x_i$, and finally we return by the same route (see Fig.~\ref{Fig:label map}). 

The product $\delta_y=b_{(y,x_1)}\cdots b_{(y,x_k)}$  of these elements (where $k$ is the number of {\it distinct} points $x_i$) is a loop that completely surrounds $\mathcal{H}$ over the point $y\in Y$. It is easily seen to be homotopic to $\delta$ (a consequence of the monicity of the discriminant \eqref{Eq:Discriminant monic}), and hence its product structure defines a factorization of $c$ via the fixed surjection $\bm\pi$:
\begin{defn}\cite[Defn.~6.9 and 7.14]{bes-kpi1}\label{Defn:rlbl_def}
There is a labeling map $\op{rlbl}$, which to each point $y\in Y$ assigns a factorization $c=c_1\cdots c_k$ of the Coxeter element, where the factors $c_i:=\bm\pi(b_{(y,x_i)})$ are defined as above. We write $\op{rlbl}(y)=(c_1,\cdots,c_k)$.
\end{defn}

As is implicit in the previous statement, the labeling map does not depend on our choice of the path $\theta$ in $Y$. In fact, all resulting paths $\beta_{\theta}$ will be homotopic, according to Bessis' fat basepoint trick \cite[Appendix~A]{bes-kpi1}.

\section{The LL map and the trivialization theorem}
\label{Section: The LL map and the trivialization theorem}

In the previous section, we described a way to produce factorizations of the Coxeter element by intersecting the discriminant hypersurface $\mathcal{H}$ with the slices $L_y$. The most striking fact of this theory though, is that this geometric construction is sufficient to produce {\it all} block factorizations of the Coxeter element $c$. In fact, if we additionally keep track of the intersection $L_y\bigcap \mathcal{H}$, each such factorization is attained exactly once (Section~\ref{Section: The trivialization theorem}).

The geometric object that keeps track of the point configurations $L_y\bigcap \mathcal{H}$ for the various $y\in Y$ is the Lyashko-Looijenga morphism. Its natural target $E_n$ is the set of {\color{blue} centered configurations} of $n$ unordered, not necessarily distinct\footnote{This would be a subset of what topologists usually call the $n^{\op{th}}$ {\it symmetric product} of $\CC$, as usually $\op{Conf}_n(X)$ is meant to assume that points in the configuration are distinct.} points in $\CC$, i.e.,
\begin{equation}\label{Eq: Defn E_n}E_n:=\sym_n\backslash H_0,\text{ where } H_0=\big\{ (x_1,\cdots,x_n)\in \CC^n\ |\ \sum_{i=1}^n x_i=0\big\} \cong \CC^{n-1}.\end{equation}
\noindent The {\it centered} condition is due to the coefficient of $t^{n-1}$ being equal to $0$ in \eqref{Eq: Discrim. eqtn}. We write $\reg{E_n}$ for those centered configurations where the points $x_i$ are distinct.

\begin{defn}\cite[Defn.~5.1]{bes-kpi1}\label{Defn: LL map}
For an irreducible well-generated complex reflection group $W$, we define the {\it \color{blue} Lyashko-Looijenga} map by:
\begin{alignat*}{3}
& \ Y &&\xrightarrow{\quad LL \quad}\ \ & &E_n\\
 y=(f_1,&\cdots,f_{n-1}) &&\xrightarrow{\quad \quad\quad}\ \ & \text{ multiset of ro}&\text{ots of } \big( \Delta(W,{\bm f});(y,t)\big)=0
\end{alignat*}
and denote it by $LL$. Notice that there is a simple description of $LL$ as an algebraic morphism. Indeed, the (multiset of) roots of a polynomial is completely determined by its coefficients, therefore we can express $LL$ as the map:
\begin{alignat*}{3}
Y&\cong \CC^{n-1} &&\xrightarrow{\quad LL\quad } &\quad E_n &\cong\CC^{n-1}\\
y=(f_1,&\cdots,f_{n-1}) &&\xrightarrow{\quad\quad\quad} &\quad \big( \alpha_2(f_1,\cdots,f_{n-1}),&\cdots,\alpha_n(f_1,\cdots,f_{n-1})\big)
\end{alignat*}
where the $\alpha_i$'s are as in \eqref{Eq: Discrim. eqtn} and \eqref{Eq:Discriminant monic}; in particular, $LL$ is a quasi-homogeneous. 

\noindent We will write $LL(y)=\{x_1,\cdots,x_k\}$ to indicate that the natural target of $LL$ is an unordered configuration space, but we will always index the $x_i$'s with complex lexicographic order, so as to be compatible with the $\op{rlbl}$ map. As the notation suppresses the multiset data, we define $\op{mult}_{x_i}\big(LL(y)\big)$ to be the multiplicity of $x_i$ in the multiset $LL(y)$.
\end{defn}

\subsection{Geometry of the \texorpdfstring{$LL$}{LL} map}
\label{sec: geom of LL}

There is a deep and beautiful connection between geometric properties of the $LL$ map and the combinatorics of block factorizations of $c$, and our further work strongly relies on it. As a full presentation of the theory is unrealistic here (but we do give a more detailed description of its connection with the classical $\mathcal{LL}$ map on polynomials, in Section~\ref{sec:LL_on_polys}), we will instead try to give a {\it meaningful summary} of the arguments and techniques that appear. 

The bulk of the results in this section appear also in \cite{bes-kpi1} (and we provide the reference when available). However, in the proofs of some of them there are non-trivial gaps (in particular, in the proof of Prop.~\ref{Prop:labels are parabolic cox elts}, see also \cite[Prop.~39, Rem.~40, Corol.~66]{theo-thesis}) and in one important case, the proof of the finiteness property for the $LL$ map relies on a faulty geometric argument, see Lemma~\ref{Lem: LL is finite}. We show in our presentation how to clarify or fix these points and in addition, we hope to elucidate---at least justify---the various lemmas and propositions from \cite{bes-kpi1} that we will use in the rest of the paper. We refer the reader to \cite[\S5-7]{theo-thesis} for a much more detailed presentation.

\subsubsection{Transversality of the slice \texorpdfstring{$L_y$}{Ly}}
\label{Section:Transversality of the slice}

Certain features of the labeling construction are particularly relevant to the geometry of the LL map. Chief of those is the {\color{blue} transversality of the slice} $L_y$ on the discriminant hypersurface $\mathcal{H}$.
This means that $L_y$ is never part of the tangent cone (see \cite[Geometric Interlude No.~1]{theo-thesis}) of $\mathcal{H}$ at any point $(y,x)$; equivalently, it means that for all points $(y,x_i)\in L_y\cap\mathcal{H}$, the multiplicity of $x_i$ as a root of the equation  $\big(\Delta(W,\bm f); f_n\big)|_y=0$ is equal to the multiplicity $\op{mult}_{(y,x_i)}(\mathcal{H})$ of $(y,x_i)$ as a point of the discriminant $\mathcal{H}$.

\begin{lem}
The slice $L_y$ is transverse to the discriminant hypersurface $\mathcal{H}$ for all points $y\in Y$.    
\end{lem}
\begin{proof}
Indeed, let's assume on the contrary that there is some point $y\in Y$ and some  $x_i\in \mathbb{C}$ such that $(y,x_i)\in L_y\cap\mathcal{H}$ and for which the multiplicity of $x_i$ as a root of the discriminant at $y$ is greater than $\op{mult}_{(y,x_i)}(\mathcal{H})$ (it can never be smaller). Consider now the label $\op{rlbl}(y)=(c_1,\ldots,c_k)$ and let $c_i=\op{rlbl}(y,x_i)$. Now, write $[X^{\op{reg}}]$ for the parabolic stratum of the discriminant hypersurface $\mathcal{H}$ that contains the point $(y,x_i)$. We know from Lemma~\ref{Lemma: mult=codim} that $\op{mult}_{(y,x_i)}(\mathcal{H})=\op{codim}(X)$. We also know by the local embedding of braid groups $B(W_Z)\hookrightarrow B(W)$ from Section~\ref{Section: The parabolic stratification of H} that $c_i$ must belong to some parabolic subgroup $W_{X'}$ with $X'\in[X]$; in particular, $\op{codim}(V^{c_i})\leq \op{codim}(X')$.

Now, the total multiplicities of the roots of the equation $\big(\Delta(W,\bm f); f_n\big)|_y=0$ must be $n$, and our assumption on $x_i$ implies the strict first inequality:
\[
n>\sum_{x_j\in L_y\cap\mathcal{H}}\op{mult}_{(y,x_j)}(\mathcal{H})\geq \sum_{j=1}^k\op{codim}(V^{c_j}),
\]
where the second weak inequality is again an application of the local embedding of braid groups and Lemma~\ref{Lemma: mult=codim} as above. This means in particular that the product $c=c_1\cdots c_k$ must have a non-trivial fixed space, since $\op{codim}(V^c)\leq\sum\op{codim}(V^{c_j})<n$.

This is impossible; the Springer theory of regular elements completely determines the eigenvalues of a Coxeter element, and none can be equal to $1$ (see \cite[Lem.~7.2]{bes-kpi1}).
\end{proof}

A first consequence of the transversality of the slice is an explicit characterization of the labels $c_i=\bm\pi(b_{(y,x_i)})$ that are associated to the points $(y,x_i)$. Again, the embedding of braid groups $B(W_Z)\hookrightarrow B(W)$ from Section~\ref{Section: The parabolic stratification of H} allows us to compare the part of the loop $b_{(y,x_i)}$ that surrounds the point $(y,x_i)$ with the special loop $\delta_Z\in B(W_Z)$ that corresponds to the Coxeter elements of $W_Z$.

Now, these are not identical and it is not {\it a priori} the case that they must be homotopic (but this was \emph{incorrectly} claimed in the proof \cite[Lem.~7.4]{bes-kpi1} of the following Lemma, see \cite[Rem.~40]{theo-thesis}). However, the transversality of the slice implies that both loops run sufficiently far from the tangent cone of $\mathcal{H}$ at $(y,x_i)$ and this fact is enough to construct a homotopy between them. The rest of $b_{(y,x_i)}$ still preserves the conjugacy class of $c_i=\bm\pi(b_{(y,x_i)})$ so then we have a corrected proof \cite[Prop.~39]{theo-thesis} of the following.

\begin{prop}\label{Prop:labels are parabolic cox elts}
For any point $(y,x_i)\in \mathcal{H}$, the label $c_i:=\op{rlbl}(y,x_i)$ is a parabolic Coxeter element. In fact, it is a Coxeter element of a parabolic subgroup $W_Z$ for which $(y,x_i)\in [Z^{\op{reg}}]$.
\end{prop}

\subsubsection{Finiteness of the \texorpdfstring{$LL$}{LL} map and quasi-homogeneity}
\label{Section: Finiteness of the LL map and quasi-homogeneity}

The most important application of the transversality property is that $LL$ is a {\color{blue} finite morphism} (i.e. it induces a finite extension of algebras). This is a loaded algebro-geometric concept which, in our case, simplifies the degree calculation, implies flatness (in particular, the fibers of $LL$ are equi-dimensional), and guarantees various topological properties (such as Corol.~\ref{Corol: Path lifting property of LL}).

There is a known (see the exposition in \cite[Prop.~46]{theo-thesis}) criterion for a quasi-homogeneous morphism $f:\CC^n\rightarrow \CC^n$ to be finite; it has to satisfy $f^{-1}({\bm 0})={\bm 0}$. In our setting, $LL(y)={\bm 0}$ implies that the slice $L_y$ must intersect the discriminant $\mathcal{H}$ at the {\it single} point $(y,0)$, and hence with multiplicity $n$. Since the intersection is transverse, $(y,0)$ must also be of multiplicity $n$ in $\mathcal{H}$. The only such point however is the origin ${\bm 0}$ (again a consequence of the parabolic stratification (see Section~\ref{Section: The parabolic stratification of H})). This completes a proof (see also \cite[Thm.~51]{theo-thesis} for more details) of the following Lemma, correcting its original version in \cite[Thm.~5.3]{bes-kpi1}. 

\begin{lem}\label{Lem: LL is finite}
The $LL$ map is a finite morphism.
\end{lem}

The mistake in \cite{bes-kpi1} lies in the proof of the transversality property of the slize $L_y$. It is claimed (see \cite[Lem.~5.6]{bes-kpi1}) that the transversality is a corollary of the fact that the tangent cone of a hypersurface (in this case $\mathcal{H}$) is a closed subvariety of the tangent bundle of the ambient space. This last statement is always true for the \emph{Whitney cone $C_4$} of a hypersurface, see \cite[\S~3]{whitney}, but not always true for its \emph{tangent cone}; the simplest counter-example is Whitney's umbrella $y^2-zx^2$, see also \cite[Rem.~38]{theo-thesis}. Our discussion in the previous Section~\ref{Section:Transversality of the slice} is necessary to circumvent this.

Now, a finite morphism $f:\CC^n\rightarrow \CC^n$ has finite fibers and the size of the generic fiber of $f$ is an important invariant, called the {\color{blue} degree} of $f$. Because $LL$ is quasi-homogeneous, Bezout's theorem  \cite[Thm.~5.1.5]{LZ-graphs-on-surfaces} allows us to easily calculate its degree via the weights (first equality below). Using the algebraic expression for $LL$ in Defn.~\ref{Defn: LL map}, we have:
\begin{equation}
\op{deg}(LL)\ =\ \dfrac{\prod_{i=2}^n\op{deg}(\alpha_i)}{\prod_{i=1}^{n-1}\op{deg}(f_i)}\ =\ \frac{\prod_{i=2}^nih}{\prod_{i=1}^{n-1}d_i}\ =\ \frac{h^{n-1}n!}{\frac{|W|}{h}}=\frac{h^nn!}{|W|}.\label{EQ: degree of LL via Bexout}
\end{equation}

Now, the quasi-homogeneity property again implies that some abstract algebro-geometric characteristics of finite morphisms have simple topological interpretations for $LL$ (see \cite[Rem.~50]{theo-thesis} and the preceding discussion). In particular, $LL$ is open and {\it proper} (the preimages of compact sets are compact). This second topological property has the following path lifting application (Corol.~\ref{Corol: Path lifting property of LL}); it allows us to justify various intuitive arguments that involve perturbing the resulting configuration $LL(y)$ (as in Prop.~\ref{Prop:codim=l_R=mult=mult} and all of Section~\ref{Section: The trivialization theorem}).

\begin{cor}\cite[Corol.~53]{theo-thesis}\label{Corol: Path lifting property of LL}
Any path (continuous movement) in the centered configuration space $E_n$ can be lifted to a (not necessarily unique) path in $Y$. In particular, $LL$ is surjective.
\end{cor}

\subsubsection{Compatibilities between  \texorpdfstring{$LL$, $\op{rlbl}$, $\mathcal{A}_W$, and $\mathcal{H}$}{LL, rlbl, AW, and H}}

The labeling map $\op{rlbl}$ and the $LL$ morphism are both defined on the same domain $Y$. In the former, the labels $c_i:=\bm\pi(b_{(y,x_i)})$ correspond to loops around points in the intersection $L_y\bigcap\mathcal{H}$, while the latter records the $f_n$-coordinates of those points. Together, and with respect to the parabolic stratifications (see Section~\ref{Section: The parabolic stratification of H}), these two maps satisfy the following compatibility properties that are, along with Prop.~\ref{Prop:labels are parabolic cox elts}, fundamental for the next sections (see the cited proof or the longer exposition in \cite[Corol.~57]{theo-thesis} and its references).

\begin{prop}[{\cite[Prop.~8.4 and Corol.~5.9]{bes-kpi1}}]\label{Prop:codim=l_R=mult=mult}\ \newline
Let $LL(y)=\{x_1,\cdots,x_k\}$, $\op{rlbl}(y)=(c_1,\cdots,c_k)$, and $Z_i$ such that $(y,x_i)\in [Z_i^{\op{reg}}]$. Then,
$$\op{codim}(Z_i)=l_{\mathcal{R}}(c_i)=\op{mult}_{x_i}\big(LL(y)\big)=\op{mult}_{(y,x_i)}(\mathcal{H}).$$
\end{prop}
\begin{proof}[Sketch of the proof]
The first and last terms are equal by Lemma~\ref{Lemma: mult=codim}, so we only need to prove the three inequalities:
\[
\op{codim}(Z_i) \leq l_{\mathcal{R}}(c_i) \leq \op{mult}_{x_i}\big(LL(y)\big) \leq \op{mult}_{(y,x_i)}(\mathcal{H}).
\]
For the initial inequality, it is true more generally that $\op{codim}(V^g)\leq l_{\mathcal{R}}(g)$ for any $g$ in a reflection group \cite[Prop.~2.11]{reiner-jia-lewis-absolute-order} and $[V^{c_i}]=[Z_i]$ by Prop.~\ref{Prop:labels are parabolic cox elts}.

For the second inequality, it is sufficient to express $c_i$ as a product of $n_i:=\op{mult}_{x_i}\big(LL(y)\big)$ reflections (see Defn.~\ref{Defn: LL map} for the notation). By Corol.~\ref{Corol: Path lifting property of LL} there is a path in the base space $Y$ along which, in the image $LL(y)$, the point $x_i$ ``blows up" into $n_i$ distinct simple points. The map $\op{rlbl}$ will label these by reflections (for instance by Prop.~\ref{Prop:labels are parabolic cox elts}), while their product will stay equal to $c_i$ if the perturbation is small enough.

Finally, the last relation is already an equality by the transversality of the slice (see Section~\ref{Section:Transversality of the slice}). Indeed, since the point $(y,x_i)$ lies in a transverse intersection of a line and a hypersurface, its multiplicity {\it as a point of the intersection} $L_y\bigcap\mathcal{H}$ equals its multiplicity {\it as a point of the discriminant} $\mathcal{H}$ (see \cite[Corol.~12.4]{fulton-intersection-theory}). 
\end{proof}

\begin{cor}
For any point $y\in Y$, the label $\op{rlbl}(y)$ is a block factorization of $c$.
\end{cor}

\subsection{The trivialization theorem}  
\label{Section: The trivialization theorem}

To be able to study the combinatorics of block factorizations via the $\op{rlbl}$ map, we need to know how the label $\op{rlbl}(y)$ is affected as $y$ varies in the base space $Y$. In fact, some (often all) of this information is contained in the image $LL\big(y(t)\big)$ of the variation (path) $y(t)$; this is described in detail with the {\color{blue} Hurwitz rule} \cites[Lem.~6.15]{bes-kpi1}[Lem.~41]{theo-thesis}.

As a matter of fact, there is more we can do. The path lifting property of the LL map (Corol.~\ref{Corol: Path lifting property of LL}) allows us to use the space $E_n$ as input, and study how perturbing the point configurations there will affect the labels. The situation is particularly nice when the multiplicities of the points in a path $\gamma(t):[0,1]\rightarrow E_n$ are constant. 

Assume indeed that there are $k$ distinct points $x_i$ and thus identify $\gamma$ with a braid with $k$ strands. Notice that we can use the natural (complex-lexicographic) orderings of the points $x_i$ at times $t=0$ and $t=1$ to map the braid $\gamma$ to an element $g\in B_k$ of the braid group on $k$ strands, even when the configurations $\gamma(0)$ and $\gamma(1)$ are different.

Consider now a lift (under $LL$, to the space $Y$) of the path $\gamma$, that starts at some point $y\in LL^{-1}(\gamma(0))$ and runs until some other point $y'$ which we denote $\gamma\cdot y$. Note that this is not a Galois action and that we do not {\it a priori} assume that $y'$ is unique; we will call it the {\it path lifting} action of $\gamma$. 

\begin{lem}[{\cite[Corol.~6.20]{bes-kpi1}}]\label{Lem: rlbl is equivariant Hurwitz-path lifting}
The labelling map is equivariant with respect to the Hurwitz action (see Defn.~\ref{Defn: The (algebraic) Hurwitz action}) and the path lifting action. That is,
$$\op{rlbl}(\gamma\cdot y)=g*\op{rlbl}(y),$$where $\gamma$ and $g$ are as above.
\end{lem}
\begin{proof}

Any loop $\gamma\in\reg{E_n}$ can be decomposed as a sequence of moves that only affect two neighboring points, changing their relative positions. In this case, the corresponding braid $g\in B_n$ is just one of the canonical generators $s_i$ whose action is given as in Defn.~\ref{Defn: The (algebraic) Hurwitz action}.

The following Figure~\ref{fig-hurw-action} describes the effect of such a Hurwitz move on the labeling map. The loops $\beta_i$ in Fig.~\ref{fig-hurw-act-a} are the ones we used in Section~\ref{Section: The labeling map rlbl} to define the labels. We write $\op{rlbl}(y)=(\beta_1,\beta_2)$ forgetting the surjection $\bm\pi:B(W)\twoheadrightarrow W$.

The next two Figures~\ref{fig-hurw-act-b} and~\ref{fig-hurw-act-c} show the slice $L_{\gamma\cdot y}$ and on it are drawn two pairs of loops. The path lifting propert of the $LL$ map (Corol.~\ref{Corol: Path lifting property of LL}) guarantees that the blue loops $(\beta_1,\beta_2)$ are homotopic (notice that the homotopy happens inside $W\setminus V^{\op{reg}}$) to those in Fig.~\ref{fig-hurw-act-a}. The red ones $(\beta_1',\beta_2')$, on the other hand, are those assigned by the labeling map. As we can see, we have $$s_1^{-1}*\op{rlbl}(\gamma\cdot y)=s_1^{-1}*(\beta_1',\beta_2')=(\beta_2',\beta_2'^{-1}\beta_1'\beta_2')=(\beta_1,\beta_2)=\op{rlbl}(y).$$
\begin{figure}[h]
\centering
\begin{subfigure}[t]{0.3\textwidth}
\centering
\includegraphics[width=0.9\linewidth]{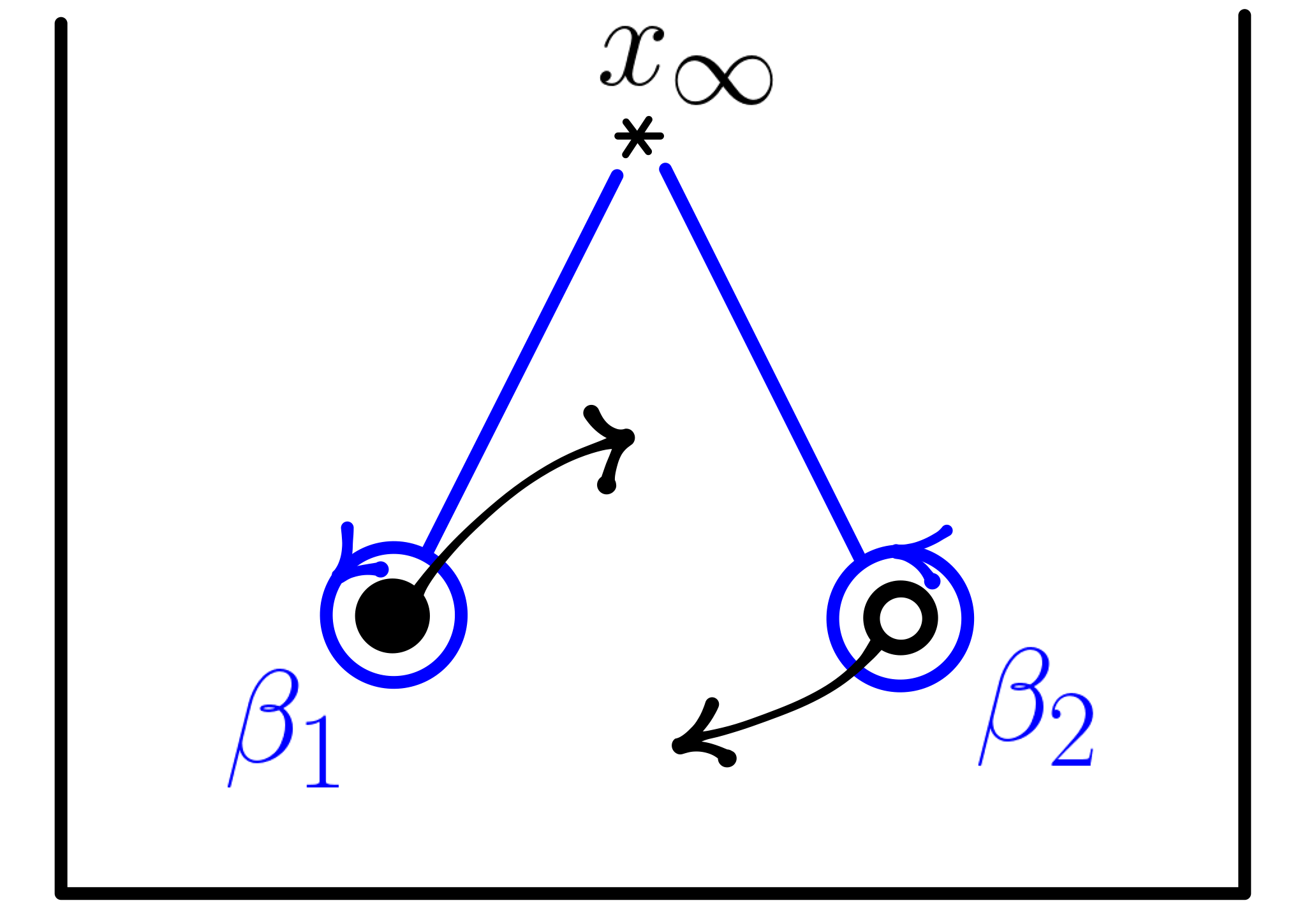}
\caption{\footnotesize As the two points in $LL(y)$ move around each other...}
\label{fig-hurw-act-a}
\end{subfigure}\quad
\begin{subfigure}[t]{0.3\textwidth}
\centering
\includegraphics[width=0.9\linewidth]{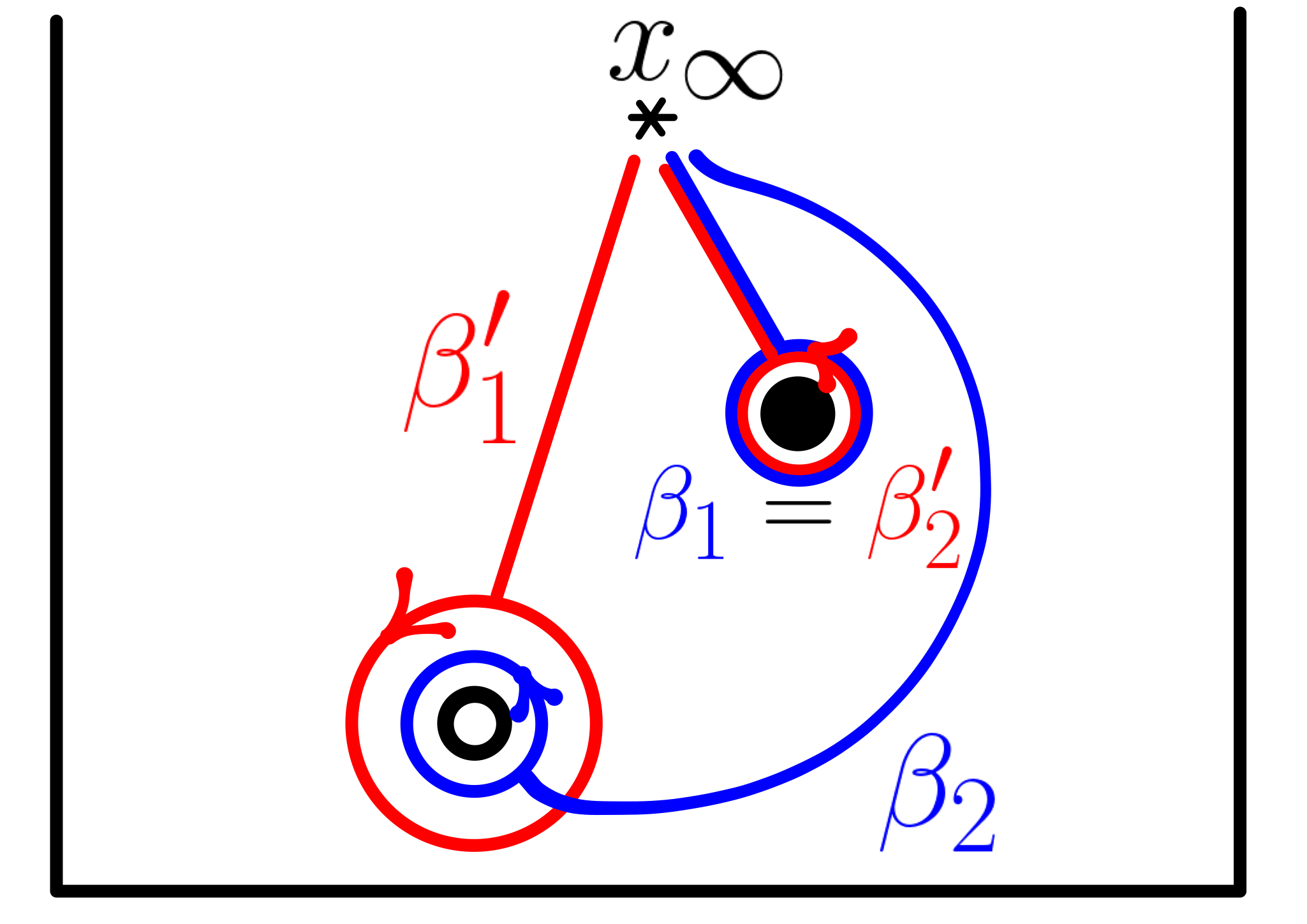}
\caption{\footnotesize ...the loop $\beta_2$ streches to avoid $\beta_1$.}
\label{fig-hurw-act-b}
\end{subfigure}\quad
\begin{subfigure}[t]{0.3\textwidth}
\centering
\includegraphics[width=0.9\linewidth]{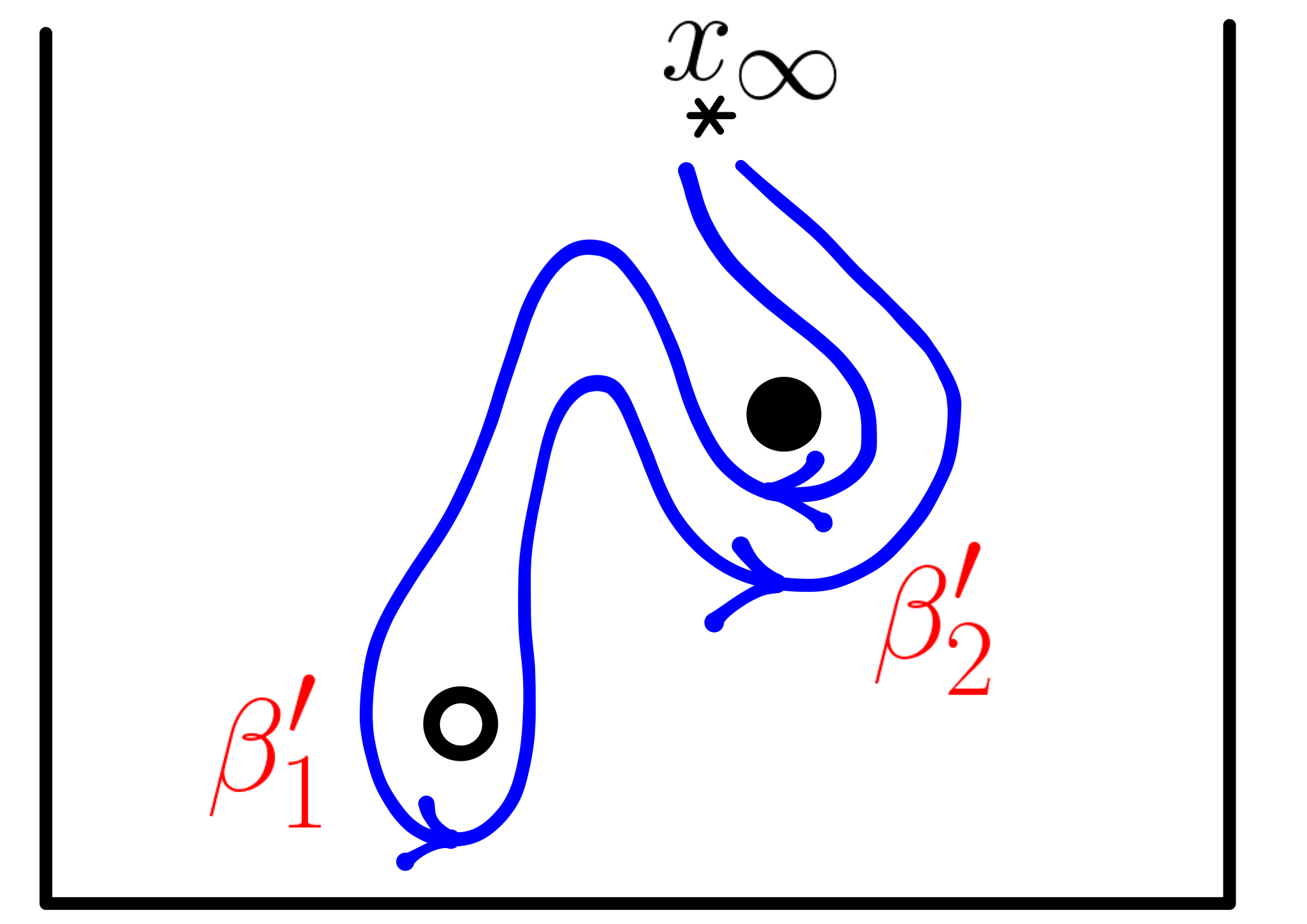}
\caption{\footnotesize Our previous loop $\beta_2$ is now homotopic to $\beta_2'^{-1}\beta_1'\beta_2'$.}
\label{fig-hurw-act-c}
\end{subfigure}
\caption{The Hurwitz action.}
\label{fig-hurw-action}
\end{figure}
\end{proof}

\subsubsection{The generic case}

The driving (combinatorial) force behind the trivialization theorem (Thm.~\ref{thm: trivialization}) is the following Proposition. The proof is uniform (via the combinatorics of chromatic pairs) for real reflection groups \cite[see][Prop.~1.6.1]{bes-dual-braid}, and case-by-case for well generated $W$.

\begin{prop}\cite[Prop.~7.6]{bes-kpi1} \label{Prop: The Hurwitz action is transitive}
The Hurwitz action is transitive on $\op{Red}_{\mathcal{R}}(c)$.
\end{prop}

Assume now that we start with a configuration $e\in E_n^{\op{reg}}$ (that is, $e$ has $n$ {\it distinct} points $x_i$, see \eqref{Eq: Defn E_n}), and a point $y\in LL^{-1}(e)$ in its preimage. The label $\op{rlbl}(y)$ will be a {\it reduced reflection factorization} of $c$ by Prop.~\ref{Prop:codim=l_R=mult=mult}. Now, the transitivity of the Hurwitz action and Lemma~\ref{Lem: rlbl is equivariant Hurwitz-path lifting} imply that {\it all} reduced reflection factorizations of $c$ appear as labels of points in the fiber $LL^{-1}(e)$.

At this point, we are ready to introduce the one (slightly) unsatisfying aspect of this theory. It turns out, that each such reduced reflection factorization of $c$ will appear as the label of a point $y\in LL^{-1}(e)$ {\it exactly once}. However, we have no ``good" reason for this fact; its proof is based on the observation \cite[Prop.~7.6]{bes-kpi1} that the size of the generic fiber $LL^{-1}(e)$ (the degree of $LL$) and the size of the set $\op{Red}_{\mathcal{R}}(c)$ happen to be equal:

\begin{prop}[The numerological coincidence]\label{Prop: Numerological coincidence}
The degree of the $LL$ map equals the number of reduced reflection factorizations of a Coxeter element $c$:
$$\op{deg}(LL)=\dfrac{h^nn!}{|W|}=|\op{Red}_{\mathcal{R}}(c)|.$$
\end{prop}

As we mentioned in the introduction (after Thm.~\ref{Thm: W-Hurwitz number}), the second equality above has a uniform proof only for real reflection groups. It is still an open problem to find a proof that works for all types, and even better one that relies on the geometry of the $LL$ map; a promising approach is via the theory of Frobenius manifolds as in the work of Hertling and Roucairol, where an equivalent statement \cite[Thm.~7.1]{hertling} is proven for the simply laced finite Coxeter groups in the context of ADE singularities. 

The previous discussion and the surjectivity of $LL$ (Corol.~\ref{Corol: Path lifting property of LL}) are enough to show that the map $$LL\times\op{rlbl}:Y^{\op{reg}}\rightarrow E_n^{\op{reg}}\times \op{Red}_{\mathcal{R}}(c),$$ is a bijection (here we simply define $Y^{\op{reg}}:=LL^{-1}(E_n^{\op{reg}})$). This is, in a sense, the {\it generic} version of the trivialization theorem (Thm.~\ref{thm: trivialization}). 

\subsubsection{The general case}

To extend the bijection to all of $Y$, we first need to introduce a suitable target space. We denote by $E_n\boxtimes D_{\bullet}(c)$ the set of compatible pairs, of a configuration $e\in E_n$, and a block factorization $\sigma\in D_{\bullet}$  (see \eqref{EQ: block factorizations}). Following \cite[Defn.~7.17]{bes-kpi1} we define (recall that $n_i$ denotes the multiplicity of $x_i$ in $e$):
$$E_n\boxtimes D_{\bullet}(c):=\Big\{ \big(\underbrace{(x_1,\cdots,x_k)}_{e\in E_n},\underbrace{(w_1,\cdots,w_l)}_{\sigma\in D_{\bullet}(c)}\big)\in E_n\times D_{\bullet}(c)\ |\ k=l\text{ and } n_i=l_{\mathcal{R}}(w_i)\ \Big\}.$$

Assuming now the generic case, we can construct arbitrary block factorizations by perturbing $y\in Y$ in such a way that a selection of the points $x_i\in LL(y)$ collide and thus, their labels merge into the block factors $w_i$. The uniqueness of this construction relies on some local properties of the $LL$ map (openness), but also on a subtle fact \cite[Corol.~66]{theo-thesis} about reflection factorizations of a parabolic Coxeter element $c_Z$; namely, it suffices to use the reflections in $W_Z$ (i.e. $\op{Red}_{\mathcal{R}(W_Z)}(c_Z)=\op{Red}_{\mathcal{R}(W)}(c_Z)$). We state the theorem \cite[Thm.~7.20]{bes-kpi1}:

\begin{thm}[Trivialization Theorem]\label{thm: trivialization}
The map $LL\times \op{rlbl}:Y\rightarrow E_n\boxtimes D_{\bullet}(c)$ is a bijection.
\end{thm}

An immediate consequence of the trivialization theorem and Prop.~\ref{Prop:labels are parabolic cox elts} is the following characterization of parabolic Coxeter elements. Bessis had given a uniform (combinatorial) proof for real $W$ in \cite[Lem.~1.4.3]{bes-dual-braid} and had originally checked it case-by-case for well-generated groups. He then gave a uniform geometric argument in \cite{bes-kpi1}, some details of which we emend in \cite[Prop.~39 and Rem.~40]{theo-thesis}.

\begin{cor}\label{Corol: parabolic equals noncrossin}
For an irreducible, well-generated, complex reflection group $W$, the set of parabolic Coxeter elements and the set of elements that are noncrossing, with respect to any Coxeter element $c$, coincide.
\end{cor}

\begin{remark}
In the original paper by Looijenga \cite[Conj.~(3.5)]{looij-complement-bifurc}, the main conjecture is equivalent to the numerological coincidence described in Prop.~\ref{Prop: Numerological coincidence}. It was proven the same year by Deligne (who also credits Tits and Zagier) \cite{deligne-letter-to-looijenga}, using a computer for $E8$ (the year being 1974). The proof was a uniform recursion (rediscovered in \cite[Corol.~3.1]{reading-chains-in-noncrossing}), solved case-by-case.
\end{remark}

\subsubsection{The ramification formula}\label{sec:ramific_form}

An important application of the trivialization theorem is an explicit ramification formula for the $LL$ map (Corollary~\ref{cor:ramif_form_LL}). Recall for a finite morphism $f:\mathbb{C}^n\rightarrow\mathbb{C}^n$ and a point $y\in\mathbb{C}^n$ of the domain, the definition of the {\color{blue} geometric multiplicity $\op{mult}_y(f)$} (also called \emph{ramification degree}) of $f$ at $y$: If $e:=f(y)$ and $e'$ is a generic point close to $e$, then we define $\op{mult}_y(f)$ to be the number of points in the fiber $f^{-1}(e')$ that are close to $y$ (see \cite[\S~7.2]{theo-thesis}). 

For a finite morphism $f:\mathbb{C}^n\rightarrow\mathbb{C}^n$ the sum of the geometric multiplicities $\op{mult}_y(f)$ over all points in some fiber $f^{-1}(e)$ will always equal the degree of the map $f$ (this is a consequence of the flatness of $f$). This is known as the {\color{blue} \emph{ramification formula}} for $f$ (\cite[Ex.~4.3.7]{fulton-intersection-theory} and \cite[Defn.~74]{theo-thesis}):
\begin{equation}
    \op{deg}(f)=\sum_{y\in f^{-1}(e)}\op{mult}_y(f)\quad\quad\text{for any }e\in\mathbb{C}^n.\label{eq:ramific_form}
\end{equation}
In our case, the following proposition relates the geometric multiplicity of $LL$ at some point $y$ with the combinatorics of its label $\op{rlbl}(y)$, and thus gives an explicit version of the ramification formula \eqref{eq:ramific_form}. It is a direct corollary of the trivialization theorem once it is afforded a natural topology that turns the bijection of the statement into a homeomorphism. 

\begin{prop}[{\cite[Proof of Thm.~7.20]{bes-kpi1}}]\ \newline
Let $y\in Y$ be a point in the domain of the $LL$ map and $\op{rlbl}(y)=(c_1,\ldots,c_k)$ its associated label (a block factorization). Then, the geometric multiplicity of $LL$ at $y$ is given by
\[
\op{mult}_y(LL)=\prod_{i=1}^k|\op{Red}_{\mathcal{R}}(c_i)|.
\]
\end{prop}

\begin{cor}[Ramification formula for $LL$]\label{cor:ramif_form_LL}
For any centered configuration $e\in E_n$, we have that
\[
\dfrac{h^nn!}{|W|}=\op{deg}(LL)=\sum_{y\in LL^{-1}(e)}\op{mult}_y(LL)=\sum_{(c_1,\ldots,c_k)}\prod_{i=1}^k|\op{Red}_{\mathcal{R}}(c_i)|,
\]
where the sum ranges over all block factorizations $c_1\cdots c_k=c$ of the Coxeter element $c$ that are compatible with the configuration $e$.
\end{cor}

\subsection{Concordance with the \texorpdfstring{$\mathcal{LL}$}{LL} map on polynomials}\label{sec:LL_on_polys}

In this section we describe how Bessis' definition of the $LL$ map (Defn.~\ref{Defn: LL map}) relates to the classical definition by Looijenga, later used and popularized by Arnold. Originally the $LL$ map was defined on (versal) deformations of simple singularities, as the map that sends the deformation (which is a function from some space $\mathbb{C}^k$ to $\mathbb{C}$) to its set of critical values. Now, simple singularities correspond to simply laced Weyl groups and the type-$A_{n-1}$ case is just the singularity $x^n=0$; a versal deformation for it is any monic, degree-$n$ polynomial. Later on, Arnold \cite{arnold-first-LL-application}, Lando and Zvonkine \cite{LZ-mult-of-LL}, Baines \cite{baines-LL-even-and-odd}, focused just on this type-$A$ case and developed a geometric and combinatorial theory on the $LL$ map. We follow below the presentation from \cite[\S5.1]{LZ-graphs-on-surfaces}.

\begin{defn}
Consider a monic polynomial with complex coefficients $a_i$, whose roots sum to $0$:
\begin{equation}
p(x)=x^n+a_2x^{n-2}+\cdots+a_{n-1}x+a_n.\label{eq:p(x)_for_LL}
\end{equation}
We define the {\color{blue} $\mathcal{LL}$ map on polynomials} on the set of such polynomials $p(x)$ via
\[
{\color{blue}\mathcal{LL}\Big(p(x)\Big)}=\Bigg\{
\genfrac{}{}{0pt}{0}{\text{multiset of critical}}{\text{values of }p(x)} \Bigg\}.
\]
\end{defn}

\begin{remark}
The $\mathcal{LL}$ map is naturally defined on equivalence classes of holomorphic maps; in the case of polynomials, after a M{\"o}bius transformation of the domain $\mathbb{C}$ we may assume that the roots of the polynomial sum to zero (i.e. that $a_1=0$). This does not correspond to the particular expression \eqref{Eq:Discriminant monic} for the discriminant hypersurface but more to the fact that the reflection representation of $S_n$ is the subspace $\sum_{i=1}^nx_i=0$ of $\mathbb{C}^n$.
\end{remark}

To make an explicit connection between Bessis' $LL$ map of Defn.~\ref{Defn: LL map} and the $\mathcal{LL}$ map on polynomials, we need to choose a suitable system of basic invariants for the action of $S_n$ on its reflection representation $V$. In particular, a system that satisfies \eqref{Eq:Discriminant monic} and thus guarantees that the image of the $LL$ map is a \emph{centered} configuration of points. 

We will identify the space $V$ with the subspace $\sum_{i=1}^nx_i=0$ in the ambient space $\mathbb{C}^n$. Now a point $S_n\cdot (r_1,\ldots,r_n)$ in the space $S_n\setminus V$ is naturally identified with the monic polynomial $p(x)$ with set of roots $\{r_1,\ldots,r_n\}$. Notice that $p(x)$ is of the form \eqref{eq:p(x)_for_LL}, since $\sum_{i=1}^n r_i=0$. Recall that the {\color{blue} discriminant $\op{Disc}_x\big(p(x)\big)$} of a polynomial $p(x)$ is a polynomial expression on the coefficients of $p(x)$ that equals $0$ when $p(x)$ has a double root; i.e. when $0$ is a critical value of $p(x)$. If $p(x)$ is monic, with roots $r_1,\ldots, r_n$ then\footnote{This is a symmetric expression in the roots of $p(x)$, hence a polynomial expression on the coefficients of $p(x)$.} $\op{Disc}_x\big(p(x)\big)=\prod_{i<j}(r_i-r_j)^2$, for example: 
\begin{align*}
\op{Disc}_x\big(x^3+a_2x+a_3)&=-27a_3^2-4a_2^3\quad\quad\quad\text{and}\\
\quad\op{Disc}_x\big(x^4+a_2x^2+a_3x+a_4)&=256a_4^3-128a_2^2\cdot a_4^2+(16a_2^4+144a_2a_3^2)\cdot a_4-(4a_2^3a_3^2+27a_3^4).
\end{align*}
We can express the discriminant as a polynomial in the constant term $a_n$, whose coefficients are polynomials in the $a_i$, $i=2,\ldots n-1$. Then, it has degree $n-1$ and starts with the term $(-n)^na_n^{n-1}$.

\begin{defn}\label{defn:LL_map_polys}
For a polynomial $p(x)$ as in \eqref{eq:p(x)_for_LL}, we define {\color{blue} $D\big(p(x)\big)$} as the scaled coefficient:
\[
{\color{blue} D\big(p(x)\big)}:=\dfrac{1}{(-n)^n}\cdot \Big[a_n^{n-2}\Big]\Big( \op{Disc}_x(x^n+a_2x^{n-2}+a_3x^{n-3}+\cdots+a_{n-1}x+a_n\Big).
\]
For example $D(x^3+a_2x+a_3)=0$ and $D(x^4+a_2x^2+a_3x+a_4)=\dfrac{-1}{2}\cdot a_2^2$. 
\end{defn}

The following Proposition gives an interpretation of $D\big(p(x)\big)$ in terms of critical values. For any polynomial $p(x)=x^n+a_2x^{n-2}+\cdots+a_n$, we will write ${\color{blue}p_0(x)}:=x^n+a_2x^{n-2}+\cdots+a_{n-1}x$ for its shift with $0$-constant term.
\begin{prop}\label{prop:D(p(x))_sum_crit}
The function $D\big(p(x)\big)$ equals the sum of the critical values of the polynomial $p_0(x)$.  
\end{prop}
\begin{proof}
If we treat $a_n$ as an unknown, then $\op{Disc}_x\big(p(x)\big)$ becomes a polynomial in $a_n$, every root of which is some value $c$ such that $p_0(x)+c$ has a double root. This means that $0$ is a critical value of $p_0(x)+c$ or equivalently that $-c$ is a critical value of $p_0(x)$.

The sum of the roots of this polynomial is by definition equal to $-D\big(p(x)\big)$, so that $D\big(p(x)\big)$ is then indeed the sum of the critical values of $p_0(x)$.
\end{proof}

\begin{example}
Consider the polynomial $p(x)=x^4-3/2\cdot x^2+x+3$. Then, \[
p'(x)=4x^3-3x+1=(x+1)(2x-1)^2,
\]
so that the critical points of $p(x)$ (and of $p_0(x)$) are $x=1/2$ with multiplicity $2$, and $x=-1$ with multiplicity $1$. The critical values of $p_0(x)$ are $3/16$ with multiplicity $2$, and $-3/2$ with multiplicity $1$. Finally, the sum of the critical values of $p_0(x)$ equals 
\[
\dfrac{3}{16}\cdot 2+\dfrac{-3}{2}\cdot 1=-\dfrac{9}{8}=\dfrac{-1}{2}\cdot (-3/2)^2=D\big(p(x)\big).
\]
\end{example}

We are now ready to explicitly describe the system of basic invariants for $S_n$ that satisfies \eqref{Eq:Discriminant monic}. We write $\bm r:=(r_1,\ldots,r_n)$, with $\sum_{i=1}^nr_i=0$, for the arbitrary point of $V$ and $e_i(\bm r)$ for the $i$-th elementary symmetric function on the $r_i$. We freely identify the point $\bm r$ with the monic polynomial $p(x)=x^n+e_2(\bm r)x^{n-2}+\cdot +(-1)^ne_n(\bm r)$ with roots $r_i$. Since $D\big(p(x)\big)$ is a polynomial on the coefficients of $p(x)$, it will be a symmetric function on the $r_i$. The following collection $\bm f$ of symmetric functions on the $r_i$ is a basic system of invariants for $S_n$, under which the discriminant hypersurface $\Delta(S_n,\bm f)$ takes a form as in \eqref{Eq:Discriminant monic}.
\begin{equation}
f_1(\bm r)=e_2(\bm r),\ldots, f_{n-2}(\bm r)=(-1)^{n-1}e_{n-1}(\bm r)\quad\quad\text{and}\quad\quad f_{n-1}(\bm r)=(-1)^ne_n(\bm r)+\dfrac{D\big(p(x)\big)}{n-1}.\label{eq:basic_invts_LL_poly}   
\end{equation}

The effect of this substitution in the case of $A_3=S_4$ is for instance that (the first equality is the definition of the discriminant in \eqref{eq:discr_prod_aH_eH} which agrees with the definition of the polynomial discriminant)
\[
\Delta(S_n,\bm f)=\prod_{1\leq i<j\leq 4}(r_i-r_j)^2=6912\cdot f_3^3+\big(3888f_1f_2^2-144f_1^4\big)\cdot f_3+(8f_1^6+540f_1^3f_2^2-729 f_2^4),
\]
is now a polynomial in $f_3$ whose coefficient for $f_3^2$ is $0$ as required in \eqref{Eq:Discriminant monic}.

With this selection of a basic system we can easily show that Bessis' $LL$ map of Defn.~\ref{Defn: LL map} models closely the behavior of the $\mathcal{LL}$ map on polynomials from Defn.~\ref{defn:LL_map_polys}.
Before stating the lemma, note that since $f_i(\bm r)=(-1)^{i+1}e_{i+1}(\bm r)$ for $i=1,\ldots,n-2$, we can identify the domain $Y$ of the $LL$ map with the collection of $0$-constant-term polynomials $p_0(x)=x^n+a_2x^{n-2}+\cdots+a_{n-1}x$.

\begin{prop}\label{prop:LL_concord}
With notation from the discussion above, the $LL$ map of Defn.\ref{Defn: LL map} is given via     
\[
LL\Big(p_0(x)\Big)=- \Bigg\{
\genfrac{}{}{0pt}{0}{\text{multiset of critical values of }p_0(x)}{\text{centered at }0} \Bigg\}.
\]
\end{prop}
\begin{proof}
Pick a $0$-constant-term polynomial $p_0(x)$, write $f_1,\ldots,f_{n-2}$ for the first invariants associated to it after \eqref{eq:basic_invts_LL_poly}, and write $d:=D(p_0(x))/(n-1)$ for the mean of its critical values. Now, a point $c\in\mathbb{C}$ belongs to the image $LL\big(p_0(x)\big)$ if $\Delta\big(S_n,(f_1,\ldots,f_{n-2},c)\big)=0$. This would mean that $p_0(x)+c-d$ has a double root, or equivalently it has $0$ as a critical value. This in turn means that $-c+d$ is a critical value of $p_0(x)$ and the statement follows.
\end{proof}

\begin{remark}
Note that the distinction between $0$-constant-term polynomials and arbitrary constant term polynomials in Prop.~\ref{prop:LL_concord} is immaterial: adding a constant to a polynomial only shifts its critical values by the same amount. 
To force the collection of critical values to be centered is slightly less trivial, but as we showed in Prop.~\ref{prop:D(p(x))_sum_crit}, the shift-to-centered is given by a symmetric function on the $r_i$, which is a quasi-homogeneous function of the coefficients of $p_0(x)$.
\end{remark}

Finally, let us note that the classical monodromy construction that produces factorizations in $S_n$ (called constellations in \cite{LZ-graphs-on-surfaces}) involves studying preimages of polynomials around their critical values; this is equivalent to studying the preimages $p^{-1}(0)$ for a \emph{family} of polynomials whose constant terms we vary. The latter interpretation corresponds exactly to the labeling map $\op{rlbl}$ of Section~\ref{Section: The labeling map rlbl} as the $f_n$ invariant of \eqref{eq:basic_invts_LL_poly} is just a shift of the constant term of the polynomial, and the lift of the corresponding braids in \eqref{EQ:1PBW1} is again the permutation of the roots $r_i$. 

The reason that the classical $\mathcal{LL}$ map produces factorizations of the long cycle is that the product of the monodromies around the finite critical values of a polynomial $p(x)$ equals its monodromy around infinity, where $p(x)$ behaves like $x^n$ and hence its monodromy is the long cycle. The long cycle becomes the Coxeter element in reflection groups and relating the monodromies around the finite critical values with those at infinity becomes the content of Definition~\ref{Defn:rlbl_def}.

\section{Primitive factorizations of a Coxeter element}
\label{Section: Primitive factorizations of a Cox elt}

This section is devoted to the proof of the main theorem of this paper, on the number of primitive factorizations of $c$ (Thm.~\ref{Thm: Primitive facto enume}). Its derivation follows a pattern of results from singularity theory, where suitable variants of the $LL$  map are constructed and whose degree computation helps enumerate different combinatorial objects.

A general class of natural enumeration questions, involves counting factorizations with prescribed conjugacy classes of factors. The simplest of those in our setting are the so called {\it \color{blue} primitive factorizations}; that is, block factorizations (as in \eqref{EQ: block factorizations}), of the form $$c=c_1\cdot t_1\cdots t_{n-k},$$ where $c_1$ belongs to a given conjugacy class, whose elements have reflection length $k$, and where the $t_i$'s are reflections. The term \emph{primitive} originated with the work \cite{LZ-mult-of-LL} where natural geometric strata of the discriminant were associated to such factorizations as above (but in the symmetric group $S_n$ only) and where \emph{arbitrary} block factorizations were enumerated after studying the intersections of these \emph{primitive} strata (see Thm.~\ref{Thm: LZ main result}). This same construction of primitive strata will work out for us as well, see Section~\ref{Section: The shadow stratification}, but we will only be able to count primitive factorizations (for arbitrary  block factorizations, the difficulty lies with the intersection theory of our primitive strata which is more complicated, see the discussion after Thm.~\ref{Thm: LZ main result}).

By Corol.~\ref{Corol: parabolic equals noncrossin} we know that all factors in a block factorization of $c$ have to be parabolic Coxeter elements. That is, not all conjugacy classes may appear in a block factorization of $c$, and those that do appear can be indexed by data associated to a flat $Z\in\mathcal{L}_W$. We recall here the concept of \emph{type}, which is essentially due to \cite[above Thm.~6.3]{reiner-athanasiadis-noncrossing-partitions-type-D-2004}, and generalizes the \emph{block sizes} of a partition:

\begin{defn}\label{Defn: parabolics and primitives of type Z}
We say that $c_i$ is a {\it \color{blue} parabolic Coxeter element of type $\orb{Z}$}, for an orbit $\orb{Z}\in W\backslash \mathcal{L}_W$, if $c_i$ is a Coxeter element of some parabolic subgroup $W_{Z'}$ such that $\orb{Z}=\orb{Z'}$. Notice that all such elements form a single conjugacy class in $W$; we write $c_{\orb{Z}}$ for an arbitrary representative.

\noindent Similarly, we say that $c=c_1\cdot t_1\cdots t_{n-k}$ is a {\it \color{blue} primitive factorization of type $\orb{Z}$} and we write $$c=c_{\orb{Z}}\cdot t_1\cdots t_{n-k},$$ if $c_1=c_{\orb{Z}}$ is a parabolic Coxeter element of type $\orb{Z}$. 
\end{defn}

\subsection{Lifting the Lyashko-Looijenga morphism}
\label{Section: Lifting the LL map} 

By Prop.~\ref{Prop:labels are parabolic cox elts}, we know that a point $(y,x_i)\in\mathcal{H}$ may be labeled by a parabolic Coxeter element of type $\orb{Z}$ (i.e. $\op{rlbl}(y,x_i)=c_{\orb{Z}}$) if and only if $(y,x_i)\in\orb{\reg{Z}}$. Therefore, to understand the points $y\in Y$ whose labels contain a factor of type $\orb{Z}$, we must study the restriction of the $LL$ map on the set $$\orb{\reg{Z}}_Y:=\{ y\in Y:\ L_y\cap\orb{\reg{Z}}\neq \emptyset\},$$ which is the projection on $Y$ of the stratum $\orb{\reg{Z}}$.

This might be difficult to do: {\it A priori}, $\orb{\reg{Z}}_Y$ is only a constructible set and we have little control on the ideal of its (Zariski)-closure. Instead we will consider a variant $\widehat{LL}$ of the Lyashko-Looijenga morphism, whose domain is the flat $Z$, and which has a much simpler geometry. We first introduce a generalization of our configuration space $E_n$ that is going to be the natural target of our lifted $\widehat{LL}$ map:

\begin{defn}\label{Defn: decorated config space}
We define the {\it \color{blue} decorated (centered) configuration space $E_{(k,1^{n-k})}$} to be the set of centered configurations of $n$ points in $\CC$, that are further required to include a special (decorated) point of multiplicity at least $k$. That is,
$$E_{(k,1^{n-k})}:=\sym_{n-k}\backslash H_{(k,1^{n-k})}=\sym_{n-k}\backslash\big\{ (\underbrace{x_0,\cdots,x_0}_{k\text{-times}},x_1,\cdots,x_{n-k})\in \CC^n\ | \ k\cdot x_0+\sum_{i=1}^{n-k}x_i=0\big\},$$ where the action of the symmetric group $\sym_{n-k}$ is on the last $n-k$ coordinates.

It is easy to see, via the Vieta formulas again, that $E_{(k,1^{n-k})}\cong \CC^{n-k}$. Indeed, the coefficients of the polynomial $(t-x_1)\cdots (t-x_{n-k})$ completely determine the unordered configuration $\{x_1,\cdots,x_{n-k}\}$ and the {\it centered} condition gives $x_0$. 

We will denote its elements by $\{\widehat{x_0},x_1,\cdots,x_{n-k}\}$. Notice that we are not assuming $x_0$ to be different from the $x_i$'s. As with $E_n$, we will write $\reg{E_{(k,1^{n-k})}}$ for those decorated configurations where $\widehat{x_0}\neq x_i\neq x_j,\ \forall i\neq j$.
\end{defn}

Given an $(n-k)$-dimensional flat $Z\in \mathcal{L}_W$, we may easily express the restrictions of the fundamental invariants $f_i$ on $Z$ as polynomials in a basis ${\bm z}:=(z_1,\cdots,z_{n-k})$ of $Z$. Indeed, we can choose a basis of $V$ that extends ${\bm z}$  and write the $f_i$'s with respect to that basis. Then their restrictions on $Z$ will involve no other variables but the $z_i$'s.

We may therefore parametrize $\orb{Z}_Y$ via $y({\bm z}):=\big(f_1({\bm z}),\cdots,f_{n-1}({\bm z})\big)$. Notice that by treating points in $Y$ as images of points ${\bm z}$, we gain information about the multiset $LL\big( y({\bm z})\big)$. In particular, we know that it contains the point $f_n({\bm z})$ and with multiplicity at least $\op{codim}(Z)=k$ (by Prop.~\ref{Prop:codim=l_R=mult=mult}). We are now ready to introduce the following {\it lift} of the $LL$ map:
\begin{defn}
For an irreducible well-generated complex reflection group $W$ and a flat $Z\in\mathcal{L}_W$, we define the {\it \color{blue} lifted Lyashko-Looijenga} map, denoted\footnote{In what follows, the dependence on the flat $Z$ will be suppressed for ease of notation.} $\widehat{LL}$, by:
\begin{alignat*}{3}
&Z&&\xrightarrow{\displaystyle \quad\widehat{LL}\quad}& &E_{(k,1^{n-k})}\\
{\bm z}:=(z_1,&\cdots&,z_{n-k})&\xrightarrow{\displaystyle \ \quad\quad\quad}&\text{ multiset }LL\big(y({\bm z})\big)\text{ wi}&\text{th the decorated point }\widehat{f_n({\bm z})}.
\end{alignat*}
\end{defn}

\noindent%notice the indent... has to be at the same line as begin minipage
\setlength{\currentparindent}{\parindent}%or the whole minipage is shifted one space right
\begin{minipage}{0.5\textwidth}
\setlength{\parindent}{\currentparindent}

The diagram on the right describes the relation between $LL$ and $\widehat{LL}$. It is immediate by the definition, that if $F$ is the forgetful map that sends the decorated multiset $\{\widehat{x_0},\cdots,x_{n-k}\}$ to the undecorated one $\{x_0,\cdots,x_{n-k}\}$ (respecting the multiplicity of $x_0$), the diagram commutes. That is, we have 
\begin{equation}F\circ\widehat{LL}=\big(LL\circ\op{pr}_Y\circ\rho\big)|_Z.\label{Eqn: LL and LL^}\end{equation}

\begin{remark}
Notice that $F$ is not in general invertible. If there are several points of multiplicity greater than or equal to $k$, there is no way to know which one was decorated. In fact, we should think of $E_{(k,1^{n-k})}$ as a desingularization (since it is isomorphic to $\CC^{n-k}$) of its image in $E_n$.%\footnote{Which is nothing else but the stratum $\lambda=(k,1^{n-k})$ of $\sym_n\backslash \CC^{n-1}=E_n$.}
\end{remark} 
\end{minipage}\quad
\begin{minipage}{0.45\textwidth}
\begin{tikzpicture}[baseline=(a).base]
\node[scale=1.2] (a) at (0,0){
\begin{tikzcd}
\ \ V\ \ \supset\ Z \arrow[d, "\rho" left, ->,shift left=-5]
  \arrow[d, "\rho", ->,shift left=7]  \rar{\textstyle \quad \widehat{LL} \quad}  & E_{(k,1^{n-k})} \ar[dd, "F"] \\
 W\backslash V \supset \orb{Z} \arrow[d, "\op{pr}_Y" left, ->,shift left=-5]
  \arrow[d, "\op{pr}_Y", ->,shift left=7] \\
\quad \ \ Y\ \ \supset \orb{Z}_Y  \rar{\textstyle \quad LL \quad} & E_n\\ 
\end{tikzcd}
};
\end{tikzpicture}\vspace{-1cm}\captionof{figure}{The lifted Lyashko-\newline Looijenga morphism.}\label{Fig: Lift of the LL map}

\end{minipage}

\subsection{The geometry of the lifted \texorpdfstring{$\widehat{LL}$}{LL} map}
\label{Section:The geometry of the lifted map LL}

Our first step towards understanding the geometry of the $\widehat{LL}$ map will be to give it an explicit description in terms of polynomials. We will need to study the restriction of the discriminant $\big(\Delta(W,{\bm f});(y,t)\big)$ on $\orb{Z}_Y$.

As before, we may view the $\alpha_i$'s as polynomials in ${\bm z}$. The fact that $f_n({\bm z})$ is always a root of the discriminant at $y=y({\bm z})$, and of multiplicity at least $k$, implies that we can factor the latter 
\begin{equation}\label{Eq:Discriminant polyn one}\big(\Delta(W,{\bm f});(y({\bm z}),t)\big)=t^n+\alpha_2\big(y({\bm z})\big)\cdot t^{n-2}+\cdots +\alpha_n\big(y({\bm z})\big),\end{equation} as
\begin{equation}\label{EQ: Discriminant polun two}\big(\Delta(W,{\bm f});({\bm z},t)\big)=\big( t-f_n({\bm z})\big)^k\big(t^{n-k}+b_1({\bm z})\cdot t^{n-k-1}+\cdots + b_{n-k}({\bm z})\big),\end{equation} where the $b_i$'s are {\it a priori} functions of ${\bm z}$. Our first task will be to show that they are in fact polynomials in ${\bm z}$:

\begin{lem}\label{Lemma: b_i's are homg polys in z_i's}
The coefficients $b_i({\bm z})$ that appear in the previous factorization of the discriminant $\big( \Delta(W,{\bm f});({\bm z},t)\big)$ are homogeneous polynomials in the $z_i$'s, of degree $hi$.
\end{lem}
\begin{proof}
Indeed, by comparing coefficients on the right hand sides of \eqref{Eq:Discriminant polyn one} and \eqref{EQ: Discriminant polun two}, we get the following equations:
\begin{alignat*}{1}
0&=kf_n-b_1,\\
\alpha_2&=b_2-kf_nb_1+\binom{k}{2}f_n^2,\\
\alpha_3&=b_3-kf_nb_2+\binom{k}{2}f_n^2b_1-\binom{k}{3}f_n^3,\\
\alpha_4&=\cdots,
\end{alignat*} where $\alpha_i,b_i$ and $f_n$ are all considered as functions on ${\bm z}$. Now, by definition, the $\alpha_i$'s and $f_n$ are polynomials in ${\bm z}$. Moreover, the above equations can be used to inductively express $b_i$ as a polynomial in the $\alpha_j$'s (with $j\leq i$) and $f_n$; therefore as a polynomial in the $z_i$'s. 

The homogeneity is also an immediate consequence of the previous argument. Indeed, the $\alpha_i$'s are weighted-homogeneous in the $f_i$'s, of weighted-degree $hi$ (Prop.~\ref{Prop: Discriminant is monic}). This means precisely that the $\alpha_i$'s are {\it homogeneous in the ${\it z_i}$'s} of degree $hi$. Along with the fact that $\op{deg}(f_n)=h$, this forces (inductively) all monomials that appear in the $i^{\op{th}}$ equation to be homogeneous and of degree $hi$.
\end{proof}

\begin{cor}
The lifted $\widehat{LL}$ map is an algebraic morphism, given explicitly as:
\begin{alignat*}{3}
Z&\cong \CC^{n-k}& & \xrightarrow{\displaystyle \quad\widehat{LL}\quad}\ & E_{(k,1^{n-k})}&\cong \CC^{n-k}\\
{\bm z}:=(z_1,&\cdots,z_{n-k})&&\xrightarrow{\quad \quad \quad \ }\ &\big(b_1(z_1,\cdots,z_{n-k}),&\cdots,b_{n-k}(z_1,\cdots,z_{n-k})\big).
\end{alignat*}
\end{cor}
\begin{proof}
One can see by the definition of the $b_i$'s, and our choice of parametrization for the space $E_{(k,1^{n-k})}$ as described in Defn.~\ref{Defn: decorated config space}, that the tuple $(b_1,\cdots,b_{n-k})$ represents the decorated multiset $\widehat{LL}({\bm z})$. The algebraicity of the map $\widehat{LL}$ is precisely the previous lemma.
\end{proof}

\begin{remark}
Notice that in the same way, we may show that the forgetful map $F$ is algebraic. Indeed, it is precisely given as ${\bm b}:=(b_1,\cdots,b_{n-k})\xrightarrow{F} \big(\alpha_2({\bm b}),\cdots ,\alpha_n({\bm b})\big)$, where the $\alpha_i$'s are given in terms of ${\bm b}$ according to the equations in the proof of Lemma~\ref{Lemma: b_i's are homg polys in z_i's}.
\end{remark}

We are now in a similar situation as in Section~\ref{Section: The LL map and the trivialization theorem}. Instead of attempting to reproduce all of its statements in the context of the lifted $\widehat{LL}$ map, we focus only on those that are pertinent to the enumerative questions; namely the finiteness and the degree calculation.

\begin{prop}
The lifted $\widehat{LL}$ map is a finite morphism and its degree is given by $$\op{deg}(\widehat{LL})=h^{\op{dim}(Z)}\cdot \big(\op{dim}(Z)\big)!\ .$$
\end{prop}
\begin{proof}
As $\widehat{LL}$ is homogeneous, we may apply the same criterion for finiteness as in Section~\ref{Section: Finiteness of the LL map and quasi-homogeneity}. In order to show that $\big(\widehat{LL}\big)^{-1}({\bm 0})={\bm 0}$, we rely on the connection with the $LL$ map, as described in Fig.~\ref{Fig: Lift of the LL map}. 

To begin with, notice that ${\bm 0}\in E_{(k,1^{n-k})}$ represents the multiset with $n$ copies of $0$, where the (unique) element $0$ is decorated. Of course, ${\bm 0}\in E_n$ is the same multiset without the decoration. Now, it is easy to see that $F^{-1}({\bm 0})={\bm 0}$ (this just says that a multiset with a single element can only be decorated in one way).

Therefore, if $\big( \widehat{LL}\big) ^{-1}({\bm 0})\neq {\bm 0}$, this implies that $\big(\widehat{LL}\big)^{-1}\circ F^{-1}({\bm 0})\neq {\bm 0}$, which is the same as $$ \rho^{-1}\circ\op{pr}_Y^{-1}\circ LL^{-1}({\bm 0})\neq{\bm 0},$$ according to \eqref{Eqn: LL and LL^} (here the notation is again as in Fig.~\ref{Fig: Lift of the LL map}). But we have shown already (see Section~\ref{Section: Finiteness of the LL map and quasi-homogeneity}) that $LL^{-1}({\bm 0})={\bm 0}$, and then $\op{pr}_Y^{-1}({\bm 0})={\bm 0}$, because $L_{\bm 0}$ intersects $\mathcal{H}$ only at the origin, and finally $\rho^{-1}({\bm 0})={\bm 0}$, because ${\bm 0}$ is the unique point in $Z$ fixed by the whole group $W$. That is, we must have $\big(\widehat{LL}\big)^{-1}({\bm 0})={\bm 0}$.

Our degree calculation is the same as in \eqref{EQ: degree of LL via Bexout}. Bezout's theorem gives us the formula: $$\op{deg}(\widehat{LL})=\prod_{i=1}^{n-k}\op{deg}(b_i)=\prod_{i=1}^{n-k}hi=h^{n-k}\cdot (n-k)!\ ,$$and since $\op{dim}(Z)=n-k$, the proof is complete.
\end{proof}

\subsection{Enumeration of Primitive factorizations}
\label{Section:Enumeration of primitive factorizations}

We would like now to apply this geometric analysis of the $\widehat{LL}$ map to our enumerative problem. First we will use the trivialization theorem to phrase the question in terms of the local geometry of the $LL$ map, which then we will reduce to a simpler problem for the $\widehat{LL}$ map.

Recall how Bessis' trivialization theorem (Thm.~\ref{thm: trivialization}) relates the fibers of the $LL$ map with compatible block factorizations. To enumerate {\it primitive} factorizations, we consider a special (centered) multiset $e\in E_n$ whose leftmost point is of multiplicity $k$ and whose other points are simple. For instance, pick \begin{equation}e:=\left(\underbrace{\frac{-1}{k}\binom{n-k+1}{2},\cdots,\frac{-1}{k}\binom{n-k+1}{2}}_{k-\text{times}},1,\cdots,n-k\right).\label{EQ: choice of e in E_n to be lifted}\end{equation} Now the trivialization theorem implies the following Lemma:

\begin{lem}\label{Lemma: FACT_Z and size of fiber}
Let $W$ be an irreducible, well-generated complex reflection group, $Z$ one of its flats, and let $e$ be as above. Then, if $\op{FACT}_W(Z)$ denotes the number of primitive factorizations of type $\orb{Z}$:
$$\op{FACT}_W(Z)=\# \big\{ LL^{-1}(e)\cap \orb{Z}_Y\big\}.$$
\end{lem}
\begin{proof}
Indeed, the bijectivity of the map $LL\times \op{rlbl}:Y\rightarrow E_n\boxtimes D_{\bullet}$ implies that the size of (number of distinct points in) the fiber $LL^{-1}(e)$ is equal to the total number of primitive factorizations of all possible types $[Z]$ with $\dim(Z)=n-k$; namely, all factorizations $c=c_1\cdot t_1\cdots t_{n-k}$, where $c_1$ is only required to satisfy $l_{\mathcal{R}}(c_1)=k$.

Now, as we discussed at the beginning of Section~\ref{Section: Lifting the LL map}, primitive factorizations of type $\orb{Z}$ may only appear as labels of points in $\orb{\reg{Z}}_Y$. However, it is easy to see that $LL^{-1}(e)$ intersects $\orb{Z}_Y$ only at $\orb{\reg{Z}}_Y$. This is a consequence of Prop.~\ref{Prop:codim=l_R=mult=mult} and  concludes the proof. 
\end{proof}

In the previous Lemma~\ref{Lemma: FACT_Z and size of fiber}, we  expressed the number of primitive factorizations of type $[Z]$, as in Defn.~\ref{Defn: parabolics and primitives of type Z}, in terms of the fiber $LL^{-1}(e)\cap\orb{Z}_Y$. In the following theorem, we relate the size of the fiber $LL^{-1}(e)\cap\orb{Z}_Y$ with the degree of the $\widehat{LL}$ map, to obtain a closed, uniform enumeration formula:

\begin{thm}\label{Thm: Primitive facto enume}
Let $W$ be an irreducible, well-generated, complex reflection group and let $Z$ be one of its flats. Then, the number of primitive factorizations of type $\orb{Z}$ is given by the formula $$\op{FACT}_W(Z)=\dfrac{h^{\op{dim}(Z)}\cdot \big(\op{dim}(Z)\big)!}{[N_W(Z):W_Z]},$$ where $N_W(Z)$ and $W_Z$ are, respectively, the setwise and pointwise stabilizers of $Z$.
\end{thm}
\begin{proof}
To prove the theorem, we will lift the special multiset $e\in E_n$ (as in \eqref{EQ: choice of e in E_n to be lifted}) all the way to the flat $Z$, following both paths described in Fig.~\ref{Fig: Lift of the LL map}, and then we will compare the two fibers.

\vspace{0.1cm} \noindent $\bm{ \big( \rho|_Z \big)^{-1}\circ\big( \textbf{pr}_Y|_{[Z]}\big)^{-1}\circ \big( LL|_{[Z]_Y} \big)^{-1}(e)}$: The application of the first inverse map gives us precisely the set on the right hand side of Lemma~\ref{Lemma: FACT_Z and size of fiber} above. For the second map, recall first that $\op{pr}_Y^{-1}(y)$ is by definition equal to the intersection $L_y\cap\mathcal{H}$, whose $f_n$ coordinates are recorded in $LL(y)$. 

Now, by Lemma~\ref{Lemma: mult=codim} the restricted preimage $\op{pr}_Y^{-1}(y)\cap [Z]$ can only contain points $(y,x)$ such that $\op{mult}_{(y,x)}(\mathcal{H})\geq k$. By Prop.~\ref{Prop:codim=l_R=mult=mult} we must then have that $\op{mult}_x\big( LL(y)\big)\geq k$. But since $LL(y)=e$ and $\frac{-1}{k}\binom{n-k+1}{2}$ is the {\it unique} element in $e$ of multiplicity $k$, the preimage $\big( \op{pr}_Y|_{[Z]}\big)^{-1}(y)$ consists of the single point $\big(y,\frac{-1}{k}\binom{n-k+1}{2}\big)$. 

For the quotient map $\rho$, notice to begin with that any $(y,x)\in\orb{Z}$ such that $LL(y)=e$ must belong to $\orb{\reg{Z}}$. Indeed, if that were not the case, there would be a point in $LL(y)$ of multiplicity greater than $k$. Now, by definition, each fiber of $\rho$ over $\orb{\reg{Z}}$ is a maximal set of points in $\reg{Z}$ that are in the same $W$-orbit. The size of such a set is counted precisely by the index $[N_W(Z):W_Z]$.

Putting the previous three paragraphs together, we have the following relation between fibers: $$\#\big\{ \big(\rho|_Z\big)^{-1}\circ\big(\op{pr}_Y|_{[Z]}\big)^{-1}\circ \big(LL|_{[Z]_Y}\big)^{-1}(e) \big\}=[N_W(Z):W_Z]\cdot \# \big\{ LL^{-1}(e)\cap \orb{Z}_Y\big\}.$$

\vspace{0.1cm}\noindent $\bm{\big(\widehat{LL}\big)^{-1}\circ F^{-1}(e)}$: Since there is a single point of multiplicity $k$ in $e$, this must have been the one that was decorated. That is, $F^{-1}(e)$ contains a single configuration, which we denote by $\widehat{e}$.

For $\widehat{LL}$, recall that a finite morphism is {\it unramified} (i.e. the number of preimages of a point equals the degree of the map) over a (Zariski)-open set. In our case, it is clear that $\reg{E_{(k,1^{n-k})}}$ is open in $E_{(k,1^{n-k})}$ (actually both in the Zariski and the complex topology). 

Therefore, it will intersect any other open set (since $E_{(k,1^{n-k})}$ is irreducible) and, in fact, it will intersect the set where $\widehat{LL}$ is unramified. Without loss of generality, we may assume that $\widehat{e}$ is in that intersection (since in dealing with $e$, we haven't relied on anything but the multiplicities of its elements). This means that:
$$\#\big\{ \big(\widehat{LL}\big)^{-1}\circ F^{-1}(e)\big\}=\op{deg}(\widehat{LL})=h^{\op{dim}(Z)}\cdot \big( \op{dim}(Z)\big)!\ .$$

Putting together the last two equalities and combining them with \eqref{Eqn: LL and LL^}, we immediately get $$\#\big\{ LL^{-1}(e)\cap\orb{Z}_Y\big\}=\frac{h^{\op{dim}(Z)}\cdot \big( \op{dim}(Z)\big)!}{[N_W(Z):W_Z]},$$ which, after Lemma~\ref{Lemma: FACT_Z and size of fiber}, is exactly what we need.
\end{proof}

\begin{remark}
One can compute the numbers $[N_W(Z):W_Z]$ for the exceptional groups, using the tables in \cite[Appendix~C]{orlik-terao-arrangements-of-hyperplanes}. They also appear explicitly in literature; for the real case see \cite[Tables~III-VIII]{orlik-solomon-coxeter-arrangements-symposia}, while for the complex ones see \cite[Tables~3-11]{orlik-solomon-arrangements-defined-by-unitary}.
\end{remark}

\begin{remark}\label{Remark:Eval at H and H'}
Notice that this proof essentially works for the ambient flat $Z=V$ as well. In that case, we will have a centered configuration of $n$-many points that is decorated at the point $f_n({\bm v})$ which might not belong to them. In fact, the first coordinate of the $\widehat{LL}$ map will be $f_n$ now, since the relation with $b_1$ is $0\cdot f_n=b_1$.

More interesting is what happens if we set $Z=H$, for some reflecting hyperplane $H$. We may consider representatives $H_1,\ldots, H_r$ of the $W$-orbits of hyperplanes (for irreducible well-generated complex reflection groups $W$, we must have $r\leq 2$ but we won't need this). Then, it is clear that
\[
\sum_{i=1}^r\op{FACT}_W(H_i)=|\op{Red}_{\mathcal{R}}(c)|=\dfrac{h^nn!}{|W|}.
\]
Comparing this with the formula of Thm.~\ref{Thm: Primitive facto enume}, we get the equation
\[
hn=\sum_{i=1}^r\dfrac{|W|}{[N_W(H_i):W_{H_i}]}=\sum_{i=1}^r[W:N_W(H_i)]\cdot |W_{H_i}|=\sum_{H\in\mathcal{A}}|W_H|,
\]
where in the last summation recall that $\mathcal{A}$ denotes the set of all reflection \emph{hyperplanes} of $W$. Moreover, relying on the fact that the stabilizers $W_H$ are cyclic (when $H$ is a hyperplane), it is easy to see that $\displaystyle \sum_{H\in\mathcal{A}}|W_H|=N+N^*$, so that the above calculation recovers the formula $hn=N+N^*$ as in Defn.~\ref{Defn:Cox elt}.
\end{remark}

\subsection{A parabolic recursion on Coxeter numbers}\label{ssec:parabolic_recursion_cox_nums}

After Lemma~\ref{Lemma: FACT_Z and size of fiber} and Thm.~\ref{Thm: Primitive facto enume}, we can give an explicit version in \eqref{eq:ramif_form_parab_rec} of the ramification formula (see Section~\ref{sec:ramific_form}) of the $LL$ map for the special point $e$ of \eqref{EQ: choice of e in E_n to be lifted}. This in turn will allow us to prove Thm.~\ref{thm:parab_rec_cox} below, which is a recursion relating the Coxeter number of $W$ and those of its parabolic subgroups. 

The recursion of Thm.~\ref{thm:parab_rec_cox} is a generalization of the discussion of Remark~\ref{Remark:Eval at H and H'} but before we state it we must first introduce a useful notation for Coxeter numbers of \emph{reducible} reflection groups. A reducible reflection group is always a direct product of irreducibles $W=W_1\times\cdots \times W_k$ and we define the {\color{blue} multiset of Coxeter numbers of $W$} as the collection of the Coxeter numbers $h_i$ of $W_i$, with each $h_i$ appearing with multiplicity equal to $\op{rank}(W_i)$. For instance $\{h_i(B_3\times A_3)\}=\{6,6,6,4,4,4\}$ since the Coxeter number of $B_3$ is $6$, and that of $A_3=S_4$ is $4$. Even if $W$ is irreducible, when we refer to the \emph{multiset} of its Coxeter numbers, we are considering them with multiplicities; for example $\{h_i(B_3)\}=\{6,6,6\}$.

The following theorem appeared first as \cite[Thm.~4]{theo_chapuy_jucys} with a uniform proof for all irreducible complex reflection groups $W$, relying on a spectral study of a Laplacian operator associated to $\mathcal{A}_W$. The following proof is currently uniform only for real reflection groups because it relies on the trivialization theorem, but it sheds a different light on the theorem realizing it via the ramification formula.

\begin{thm}[Parabolic recursion of Coxeter numbers via the ramification formula of $LL$]\label{thm:parab_rec_cox} For any irreducible, well generated, complex reflection group $W$ of rank $n$ and Coxeter number $h$, we have that
\[
(h+t)^n=\sum_{Z\in\mathcal{L}_W}\Big(\prod_{i=1}^{\op{codim}(Z)}h_i(W_Z)\Big)\cdot t^{\op{dim}(Z)},
\]
where $\{h_i(W_Z)\}$ denotes the multiset of Coxeter numbers for $W_Z$ as described above.
\end{thm}
\begin{proof}
In the case of the special point $e$ of \eqref{EQ: choice of e in E_n to be lifted}, the ramification formula (Corol.~\ref{cor:ramif_form_LL}) for $LL$ becomes
\[
\dfrac{h^nn!}{|W|}=\op{deg}(LL)=\sum_{[Z]\in\mathcal{L}_W/W\atop \dim(Z)=n-k}\#\{LL^{-1}(e)\cap [Z]_Y\}\cdot |\op{Red}_{\mathcal{R}}(c_Z)|, 
\]
where $c_Z$ is any Coxeter element in $W_Z$ and $k$ is the multiplicity of the unique non-simple point of $e$. Now, if $W_Z$ decomposes as a product of irreducibles as $W_Z=W_1\times\cdots\times W_r$, and we write $c_Z=c_1\cdots c_r$ for the respective factorization of $c$ (note that $l_{\mathcal{R}}(c_Z)=k$ and each $c_i$ is a Coxeter element in $W_i$) and $h_i$ for the Coxeter numbers of $W_i$, then 
\begin{align*}
    |\op{Red}_{\mathcal{R}}(c_Z)|&=\binom{l_{\mathcal{R}}(c_Z)}{l_{\mathcal{R}}(c_1),\ldots,l_{\mathcal{R}}(c_r)}\cdot \prod_{i=1}^r|\op{Red}_{\mathcal{R}}(c_i)|\\
    &=\binom{k}{l_{\mathcal{R}}(c_1),\ldots,l_{\mathcal{R}}(c_r)}\cdot \prod_{i=1}^r\dfrac{h_i^{l_{\mathcal{R}}(c_i)}l_{\mathcal{R}}(c_i)!}{|W_i|}=\dfrac{k!}{|W_Z|}\prod_{i=1}^{k}h_i(W_Z),
\end{align*}
where the first equality is because any reduced reflection factorization of $c_Z$ is a shuffle of reduced reflection factorizations of each of the $c_i$ and the second line is Thm.~\ref{Thm: W-Hurwitz number} and the definition of the numbers $h_i(W_Z)$. After this, the ramification formula of $LL$ for $e$ becomes:
\begin{equation}
\dfrac{h^nn!}{|W|}=\op{deg}(LL)=\sum_{[Z]\in\mathcal{L}_W/W\atop \dim(Z)=n-k}\dfrac{h^{n-k}(n-k)!}{[N_W(Z):W_Z]}\cdot \dfrac{k!}{|W_Z|}\prod_{i=1}^{k}h_i(W_Z).\label{eq:ramif_form_parab_rec}
\end{equation}
There are $[W:N_W(Z)]$-many flats in the $W$-orbit of $Z$, so we can rewrite the above summation (after cancelling out $h^{n-k}$ and grouping together all the factorials in the left hand side) as:
\[
h^k\cdot\binom{n}{k}=\sum_{Z\in\mathcal{L}_W\atop\dim(Z)=n-k}\prod_{i=1}^kh_i(W_Z).
\]
The two sides of the equation above are exactly the coefficients of $t^{n-k}$ in the two sides of the formula of the theorem. Therefore, the proof is complete.
\end{proof}

\subsection{When \texorpdfstring{$N_W(Z)/W_Z$}{NW(Z)/WZ} acts as a reflection group}
\label{Section When N(Z)/W(Z) is a refl group}

\noindent%notice the indent... has to be at the same line as begin minipage
\setlength{\currentparindent}{\parindent}%or the whole minipage is shifted one space right
\begin{minipage}{0.66\textwidth}
\setlength{\parindent}{\currentparindent}

In certain cases, we may derive Thm.~\ref{Thm: Primitive facto enume} directly via a degree calculation and hence avoid the overcounting argument. Indeed, if we write $C_Z:=N_W(Z)/W_Z$ for the quotient of the setwise over the pointwise stabilizer of $Z$, then by definition the $\widehat{LL}$ map has to respect $C_Z$-orbits. That is, it will factor as the composition $ LL'\circ\rho_Z$ (see Fig.~\ref{Fig: LL' map for C_Z}; the map $\nu_Z$ is the normalization of the affine variety $[Z]$), where $LL'$ may be defined analogously to Defn.~\ref{Defn: LL map}.

When $C_Z$ acts as a reflection group, by the Shephard-Todd-Chevalley theorem, the quotient $C_Z\setminus Z$ is an affine space (this is not true in general however). As in \eqref{Eq: quotient map rho}, its coordinates correspond  to the invariant polynomials $g_1({\bm z}),\cdots,g_{n-k}({\bm z})$ of the action of $C_Z$ on $Z$. Moreover, if $d_i'=\deg(g_i)$, we will have $[N(Z):W_Z]=|C_Z|=\prod_{i=1}^{n-k}d_i'$.
\end{minipage}\ \quad
\begin{minipage}{0.33\textwidth}\vspace{-1cm}
\begin{tikzpicture}[baseline=(a).base]
\node[scale=1.3] (a) at (0,0){
\begin{tikzcd}
Z \arrow[dd, "\textstyle \rho" left, ->] \arrow[rd,"\textstyle \rho_Z" above right, ->] \arrow[rr,"\textstyle \quad \widehat{LL} \quad", ->]  &[-20 pt] &[-40 pt] E_{(k,1^{n-k})}\\
& C_Z\!\setminus\! Z \arrow[ur,"\textstyle LL'" below right,->]  \\
\orb{Z}\arrow[ur,"\textstyle \nu_Z" above left, <-] &\!\!\!\!\!\!\!\! \cong W\!\setminus\! W\!\cdot\! Z \\
\end{tikzcd}
};
\end{tikzpicture}\vspace{-1.3cm}\captionof{figure}{$\widehat{LL}\!=\! LL'\circ\rho_Z$.}\label{Fig: LL' map for C_Z}

\end{minipage}

Now, by Lemma~\ref{Lemma: b_i's are homg polys in z_i's}, the $b_i({\bm z})$ are polynomials in ${\bm z}$ and since they also have to be $C_Z$-invariant (because $\widehat{LL}$ is), they will be polynomials in the $g_i({\bm z})$. That is, we can express $LL'$ as the algebraic morphism:
$$\CC^{n-k}\cong C_Z\setminus Z\ni {\bm g}:=(g_1,\cdots,g_{n-k})\xrightarrow{\displaystyle \ LL' \ } \big(b_1(\bm g),\cdots,b_{n-k}(\bm g)\big)\in E_{(k,1^{n-k})}\cong\CC^{n-k}.$$

Now, arguing exactly as in the proof of Thm.~\ref{Thm: Primitive facto enume}, but going up (see Figures~\ref{Fig: Lift of the LL map} and~\ref{Fig: LL' map for C_Z}) only until $C_Z\setminus Z$, we recover the formula via a single degree calculation:
$$\op{FACT}_W(Z)=\op{deg}(LL')=\dfrac{\prod_{i=1}^{n-k}\deg(b_i)}{\prod_{i=1}^{n-k}\deg(g_i)}=\dfrac{\prod_{i=1}^{n-k} hi}{\prod_{i=1}^{n-k} d_i'}=\dfrac{h^{\op{dim}(Z)}\cdot\big(\op{dim}(Z)\big)!}{[N_W(Z):W_Z]}.$$

One significant advantage of this approach is that it provides a natural $q$-version of the enumerative formula that can be shown to satisfy CSP's analogous to the one in \cite{theo-CSP}:
$$\op{FACT}_{W}(Z;q):=\op{Hilb}\big( (LL')^{-1}(\bm 0),q\big)=\dfrac{\prod_{i=1}^{n-k}[hi]_q}{\prod_{i=1}^{n-k}[d_i']_q}.$$

In recent work \cite[Thm.~3.5]{douglass-restricting-invariants}, Amend et al. give sufficient conditions for $C_Z$ to act as a reflection group (and a characterization for the infinite family; see their Corol.~4.8). Notice that this is not equivalent to the statement that the restricted arrangement $\mathcal{A}_W^Z$ is a reflection arrangement (see \cite[Remark~5.2]{douglass-restricting-invariants}).
\begin{remark}
If $C_Z$ does not act as a reflection group, then the morphism $LL'$ will not be as simple (i.e. it will not be a quasi-homogeneous polynomial map anymore). We can still calculate the above Hilbert series, but the answer is not as explicit; namely, it is unclear how the \emph{correct} $q$-version of the overcounting factor is related to the Hilbert series of the invariant ring $\CC[Z]^{C_Z}$.
\end{remark}

\section{A geometric interpretation of Kreweras numbers}
\label{Section: A geometric interpretation of Kreweras numbers}

The concept of type as described in Defn.~\ref{Defn: parabolics and primitives of type Z} determines a meaningful partition of the noncrossing lattice $NC(W)$. Kreweras \cite{kreweras-noncrossing-partitions} was the first to compute the block-sizes of this partition for the symmetric group. 

\begin{defn}
We define the {\it \color{blue} Kreweras numbers} for $W$, to be the numbers $$\op{Krew}_W(Z):=\#\big\{ c_i\in NC(W)\ |\ c_i\text{ is of type }\orb{Z}\big\}.$$
\end{defn}

As is the case for the total size of the noncrossing lattice, we have uniform formulas for the Kreweras numbers, but no uniform {\it proofs} of these formulas. Recall that the characteristic polynomial of a hyperplane arrangement $\mathcal{A}$ on $V$ is defined by $$\chi(\mathcal{A},t):=\sum_{Z\in\mathcal{L}_{\mathcal{A}}}\mu(V,Z)\cdot t^{\op{dim}(Z)},$$ where $\mu(V,Z)$ is the M\"obius function on the intersection lattice $\mathcal{L}_{\mathcal{A}}$, and that $\mathcal{A}^Z$ denotes the restriction of $\mathcal{A}$ on one of its flats $Z$. 

\begin{prop}\cite[essentially][Thm.~6.3]{reiner-athanasiadis-noncrossing-partitions-type-D-2004}\label{Prop: Krew formulas}
If $W$ is a Weyl group and $Z$ is one of its flats, then we have $$\op{Krew}_W(Z)=\dfrac{\chi(\mathcal{A}_W^Z,h+1)}{[N_W(Z):W_Z]}.$$
\end{prop}
\begin{proof}[Sketch of the proof]
In \cite[Thm.~6.3]{reiner-athanasiadis-noncrossing-partitions-type-D-2004}, Athanasiadis and Reiner prove that the Kreweras numbers are equal to the corresponding statistics for nonnesting partitions. Those had already been uniformly shown to adhere to the above formula; see for instance \cite[Prop.~6.6:(1)]{sommers-b-stable-ideals}.
\end{proof}

\begin{remark}
As it happens, the characteristic polynomials $\chi(\mathcal{A}_W^Z,t)$ have integer roots: $$\chi(\mathcal{A}_W^Z,t)=\prod_{i=1}^{\op{dim}(Z)}(t-b_i^Z),$$ where the $b_i^Z$ are the so called Orlik-Solomon exponents for $Z$ (first computed in \cite{orlik-solomon-coxeter-arrangements-symposia} and \cite{orlik-solomon-unitary-reflection-groups-and-cohomology}). These nice product formulas are a consequence of the fact that restricted arrangements of reflection arrangements are free (see \cite[Thm.~4.137]{orlik-terao-arrangements-of-hyperplanes} for the implication and \cite{restriction-arrangements-are-free-hoge-gerhard} for the completed case-by-case proof).
\end{remark}

In a similar fashion to our Lemma~\ref{Lemma: FACT_Z and size of fiber}, we may relate the number of noncrossing elements of type $\orb{Z}$ (i.e. the Kreweras numbers) with the size of a particular fiber of the lifted $\widehat{LL}$ map. Let $\widehat{e}_{(k,n-k)}$ be the decorated (centered) multiset $\{\widehat{-(n-k)},\underbrace{k,\cdots,k}_{(n-k)\text{-times}}\}$ (see Defn.~\ref{Defn: decorated config space}).

\begin{lem}\label{Lem: Krew via size of fiber}
The number $\op{Krew}_W(Z)$ of noncrossing elements of type $\orb{Z}$ is equal to:
$$\frac{\#\big\{\big(\widehat{LL}\big)^{-1}(\widehat{e}_{(k,n-k)})\big\}}{[N_W(Z):W_Z]}.$$
\end{lem}
\begin{proof}
The proof is analogous to the ones in Lemma~\ref{Lemma: FACT_Z and size of fiber} and Thm.\ref{Thm: Primitive facto enume}, so we present it more compactly.

By the Trivialization Theorem and the definition of $\widehat{LL}$, the points ${\bm z}$ that lie in the fiber $\big(\widehat{LL}\big)^{-1}(\widehat{e}_{(k,n-k)})$ are such that $\op{rlbl}\big(y({\bm z})\big)=(c_1,c_2)$ where $c_1$ is of type $\orb{Z}$. Since the multiplicity of the decorated point in $\widehat{e}_{(k,n-k)}$ is equal to $k=\op{codim}(Z)$, the points ${\bm z}$ will further belong to the regular part $\reg{Z}$. 

That is, for each block factorization $c_1\cdot c_2=c$, with $c_1$ of type $\orb{Z}$, we have $[N_W(Z):W_Z]$-many preimages in the fiber. On the other hand, the number of such block factorizations is precisely the number of noncrossing elements of type $\orb{Z}$ (since each $c_1$ has a unique {\it\color{blue} Kreweras complement} $c_2=c_1^{-1}c$). This completes the argument.
\end{proof}

\subsubsection*{Speculation towards a uniform enumeration of $NC(W)$}

The previous Lemma~\ref{Lem: Krew via size of fiber} suggests that we would get a uniform proof of the Kreweras formulas (Prop.~\ref{Prop: Krew formulas}) if we could show in a {\it geometric} way that $$\#\big\{\big(\widehat{LL}\big)^{-1}(\widehat{e}_{(k,n-k)})\big\}=\chi(\mathcal{A}_W^Z,h+1)=\prod_{i=1}^{\op{dim}(Z)}(h+1-b_i^Z).$$

The difficulty here is in the fact that the fiber is not reduced. Our combinatorial description of the local multiplicities of the $LL$ map can be translated of course to the lifted case, but apparently it is not sufficient. In \cite[\S~8.1]{theo-thesis}, we give a geometric reason for this local behavior of $\widehat{LL}$ when the flat $Z$ is a line; it is not clear how to extend this to the general case.

We would like to note however that, if this is successful, the formulas for the Kreweras numbers imply (in a uniform way) the formula for the size of the noncrossing lattice: $$|NC(W)|=\frac{1}{|W|}\prod_{i=1}^n(h+d_i).$$ 

Indeed, the following calculation (see \cite[Prop.~105]{theo-thesis}), which is true for all complex reflection groups $W$, implies the previous statement by setting $t=h+1$. 
\begin{prop}
Consider the {\bf Kreweras polynomials} $\op{Krew}_W^Z(t):=\dfrac{\chi(\mathcal{A}_W^Z,t)}{[N_W(Z):W_Z]}$. Then, we have $$\sum_{\orb{Z}\in W\backslash \mathcal{L}_W}\op{Krew}_W^Z(t)=\frac{1}{|W|}\prod_{i=1}^n(t+d_i-1).$$
\end{prop}

\section{The shadow stratification}
\label{Section: The shadow stratification}

The trivialization theorem suggests an obvious stratification of the base space $Y$, with respect to the labels $\op{rlbl}(y)$. In this last section we are going to define these strata and build a framework that relates the local geometry of the $LL$ map on them, with finer enumerative and structural properties of block factorizations.

Let $\sigma=(w_1,\cdots,w_k)\in D_{\bullet}$ be a block factorization \eqref{EQ: block factorizations} of the Coxeter element $c$. As we have seen in Corol.~\ref{Corol: parabolic equals noncrossin}, the $w_i$'s are parabolic Coxeter elements; their conjugacy classes are therefore associated to $W$-orbits of flats $[Z]$ (via $w_i\rightarrow V^{w_i}$). Following \cite[Defn.~1.1.7]{LZ-graphs-on-surfaces}, we introduce: 

\begin{defn}[Passport]
For a block factorization $\sigma$, the tuple $\big(\bm Z\big):=\big([Z_1],\cdots,[Z_k]\big)$ of the {\it parabolic classes} of its factors will be called the {\color{blue} passport} of $\sigma$. If we are only interested in the set of classes that appear, we write $\big\{\bm Z\big\}:=\big\{[Z_1],\cdots,[Z_k]\big\}$ for the {\color{blue} unordered passport} of $\sigma$.

For each such $\big\{\bm Z\big\}$, we now define the {\color{blue} shadow stratum} $Y_{\{\bm Z\}}\subset Y$ as the set of points $y\in Y$ whose labels $\op{rlbl}(y)$ have (unordered) passport $\big\{\bm Z\big\}$. That is,
$$Y_{\{\bm Z\}}:=\big\{ y\in Y: \ \text{if }\op{rlbl}(y)=(w_1,\cdots,w_k),\text{ then }\exists\ \pi\in S_k:\ \big[V^{w_{\pi(i)}}\big]=\big[Z_i\big]\big\}.$$
Notice that the strata $Y_{\{\bm Z\}}$ are by definition {\it disjoint}.
\end{defn}

This {\it shadow stratification} of the base space $Y$ is a refinement of its namesake, considered by Bessis \cite[below Corol.~5.9]{bes-kpi1} and Ripoll \cite[above Thm.~5.2]{ripoll-french-paper}. In the latter, the strata were indexed by partitions $\lambda$ which recorded the ranks (codimensions) of the $Z_i$. One might justify our use of the term by picturing the points $y$ living in the ``shadow" of the intersection $L_y\bigcap \mathcal{H}$ as in Fig.~\ref{Fig:label map}.

As opposed to the strata $Y_{\lambda}$ of Bessis and Ripoll, the shadow strata $Y_{\{\bm Z\}}$ are not necessarily varieties. We can however describe them as images of (quasi-affine) varieties; this means they are {\color{blue} constructible sets} \cite[Corol.~14.7]{eisenbud-commutative-algebra}. Indeed, pick representatives $Z_i$ of the classes $[Z_i]$ (one can use multiple copies of the same flat if needed), and consider the set of points 
\begin{equation}\label{Eq: Shadow strata}
(\bm z_1,\bm z_2,\cdots, \bm z_k)\in (Z_1\times Z_2\times\cdots\times Z_k)\text{ for which } \begin{Bmatrix}
f_1(\bm z_1)=\cdots =f_1(\bm z_k)\\
f_2(\bm z_1)=\cdots =f_2(\bm z_k)\\
\cdots\\
f_{n-1}(\bm z_1)=\cdots =f_{n-1}(\bm z_k)\\
f_n(\bm z_i)\neq f_n(\bm z_j)\\
\end{Bmatrix}.
\end{equation}
This defines a quasi-affine variety $V\big(\{\bm Z\}\big)\subset Z_1\times\cdots\times Z_k$ (since the last conditions define Zariski-open sets). Here the polynomials $f_i({\bm z_j})$ are given by restriction of coordinates, and of course depend on the inclusion $Z_j\subset V$.

\begin{prop}
The image of the quasi-affine variety $V\big(\{\bm Z\}\big)$ defined by \eqref{Eq: Shadow strata} under the map
$$V\big(\{\bm Z\}\big)\ni(\bm z_1,\cdots,\bm z_k)\rightarrow (f_1(\bm z_1),f_2(\bm z_1),\cdots,f_{n-1}(\bm z_1))\in Y,$$ where it doesn't matter if we use some other $\bm z_i$ instead of $\bm z_1$, is the stratum $Y_{\{\bm Z\}}$.
\end{prop}
\begin{proof}
It is an easy application of Prop.~\ref{Prop:labels are parabolic cox elts} that indeed all the points of the stratum $Y_{\{\bm Z\}}$ belong to the image of $V\big(\{\bm Z\}\big)$. That there is nothing else is only slightly trickier:

Write $y(\bm z_1)$ for the image and notice that, by definition, the points $\big(y(\bm z_1),f_n(\bm z_i)\big)$ must all belong to the discriminant hypersurface $\mathcal{H}$ (since $y(\bm z_1)=y(\bm z_i)$ for all $i$). Now, by Prop.~\ref{Prop:codim=l_R=mult=mult}, we must also have that $$\op{mult}_{f_n(\bm z_i)}\big(LL(y(\bm z_1))\big)\geq \op{codim}(Z_i),$$ for all $i$ (since $\bm z_i$ could belong to a smaller flat $Z'\subset Z_i$ of higher codimension). But on the other hand, the intersection $L_{y(\bm z_1)}\bigcap \mathcal{H}$ contains exactly $n=\sum_{i=1}^k\op{codim}(Z_i)$ points counted with multiplicity. That is, the set $\{f_n(\bm z_1),\cdots,f_n(\bm z_k)\}$ is precisely the image $LL\big(y(\bm z_1)\big)$ and the points $\bm z_i$ are all in the regular part $Z_i^{\op{reg}}$. In other words, and by Prop.~\ref{Prop:labels are parabolic cox elts}, the labels $\op{rlbl}\big(y(\bm z_1)\big)$ have (unordered) passport $\{\bm Z\}$.
\end{proof}

The exact same argument as in the previous proof serves as a characterization of passports:

\begin{cor}
A set of $W$-orbits of flats $\{\bm Z\}=\big\{ [Z_1],\cdots,[Z_k]\big\}$ is a passport if and only if they satisfy $\sum_{i=1}^k\op{codim}(Z_i)=n$ and the equations \eqref{Eq: Shadow strata} have a solution (for any choice of representatives $Z_i$). 
\end{cor}

When $W$ is a simply laced Coxeter group, Lyashko \cite[Thm.~5]{lyashko-valid-passports} determined all possible passports by using a generalization of Dynkin diagrams. His results however are still case by case. For the other real reflection groups, the answer appears in \cite{krattenthaler-Muller-decomposition-numbers}, while in the well-generated case only the infinite families have been considered \cite{ripoll-personal}.

The previous Corollary can easily be used to confirm these results, at least in the case of the classical groups, but doesn't seem to illuminate the general situation. We therefore ask:

\begin{question}
Give a coordinate free, or otherwise intrinsic, characterization of the possible passports for a well-generated complex reflection group $W$.
\end{question}

\begin{remark}
One might guess in the real case, that a tuple of parabolic types is a passport if and only if it can be realized as a partition of the vertices of the Coxeter diagram. By this we mean a partition of the simple generators in disjoint subsets, and the type of each parabolic subgroup generated by each such subset. This is in fact not the case. Already in type $A_5$, it is easy to see that $\big(A_1^3,A_2\big)$ is a valid passport, but clearly it cannot be realized by a partition of the Coxeter diagram.
\end{remark}

\subsection{Applications on enumeration}
\label{Section: Applications on enumeration}

One of the main reasons for introducing the shadow strata is that they are the natural object, on which a local analysis of the $LL$ map might give us finer enumerative information. The following proposition (compare with \cite[Thm.~3]{LZ-counting-I}) clarifies the relation between the size of fibers of the $LL$ map, and the enumeration of block factorizations, as suggested by the trivialization theorem. First, we give a somewhat unconventional definition for the degree of $LL$ when restricted to a stratum (it models \cite[Defn.~3.5]{LZ-counting-I}).

\begin{defn}
The degree of the $LL$ map on the stratum $Y_{\{\bm Z\}}$ is the number of preimages of an arbitrary point in $LL\big(Y_{\{\bm Z\}}\big)$. It is well-defined by the trivialization theorem.
\end{defn}

\begin{prop}\label{Prop: local degree of LL and Fact}
The number of block factorizations $\sigma$ with (ordered) passport $\big(\bm Z\big)$, denoted $\op{Fact}_W\big[(\bm Z)\big]$, and the degree of the $LL$ map on the stratum $Y_{\{\bm Z\}}$ are related via the equation:
$$\op{deg}(LL)|_{Y_{\{\bm Z\}}}=\dfrac{\big|\op{Aut}(|Z_1|,\cdots,|Z_k|)\big|}{\big|\op{Aut}([Z_1],\cdots,[Z_k])\big|}\cdot \op{Fact}_W\big[(\bm Z)\big],$$where $|Z_i|$ denotes the codimension of $Z_i$ (also the rank of $W_{Z_i}$), and where an automorphism of a tuple is a permutation of its entries that leaves it invariant.
\end{prop}
\begin{proof}
First of all, notice that if $\big(\bm Z'\big)$ is any permutation of the parabolic classes of $\big(\bm Z\big)$, we have $\op{Fact}_W\big[(\bm Z)\big]=\op{Fact}_W\big[(\bm Z')\big]$. This is because the two sets of factorizations are in bijection via the Hurwitz action of any braid that permutes the corresponding classes.

The statement is now a corollary of the trivialization theorem (Thm.~\ref{thm: trivialization}). Indeed, consider a point configuration $e\in E_n$ whose elements $x_i$ (ordered complex-lexicographically) have multiplicities $n_i:=|Z_i|$. Now, by the trivialization theorem, the fiber $LL^{-1}(e)$ is in bijection with block factorizations $\sigma=(w_1,\cdots,w_k)$ for which $l_{\mathcal{R}}(w_i)=n_i$.

If we further restrict on the stratum $Y_{\{\bm Z\}}$, the points in the preimage $LL^{-1}(e)\bigcap Y_{\{\bm Z\}}$ correspond to block factorizations $\sigma$ with the same {\it ordered} prescribed lengths $n_i$, but which also have {\it unordered} passport $\{\bm Z\}$. Finally, there are exactly $\displaystyle \frac{\big|\op{Aut}(|Z_1|,\cdots,|Z_k|)\big|}{\big|\op{Aut}([Z_1],\cdots,[Z_k])\big|}$ many ways to permute the terms in a passport $\big(\bm Z\big)$, respecting the ranks $|Z_i|$. 
\end{proof}

\subsection{The complete answer is known in type \texorpdfstring{$A$}{A}}
\label{Section: The complete answer is known in type A}

{\it A priori} the geometry of shadow strata might be quite complicated. However, there is hope that our arguments of Section~\ref{Section: Primitive factorizations of a Cox elt} may, to a certain degree, be extended. In fact, most shadow strata (at least when $[Z_i]\neq [Z_j]$) can be described as intersections $Y_{\{\bm Z\}}=[Z_1^{\op{reg}}]_Y\cap\cdots\cap[Z_i^{\op{reg}}]_Y$ of our {\it primitive} strata $[Z_i^{\op{reg}}]_Y$ (see Section~\ref{Section: Lifting the LL map}).

In the symmetric group case, this approach has been successful in completely determining the local degrees of the $LL$ map. The following theorem, in conjunction with Prop.~\ref{Prop: local degree of LL and Fact} gives a purely geometric proof of the celebrated Goulden-Jackson formula \cite[Thm.~3.2]{goulden-jackson-trees-and-cacti}.

\begin{thm}\cite[Thm.~5.2.2]{LZ-graphs-on-surfaces}\label{Thm: LZ main result}
When $W$ is the symmetric group $S_n$, and $\big\{ \bm Z\big\}=\big\{[Z_1],\cdots,[Z_k]\big\}$ is an arbitrary passport, the restriction of the $LL$ map to the shadow stratum $Y_{\{\bm Z\}}$ is a smooth finite mapping of degree
$$\op{deg}(LL)|_{Y_{\{\bm Z\}}}=n^{k-1}\cdot\dfrac{\big|\op{Aut}(|Z_1|,\cdots,|Z_k|)\big|}{\big|\op{Aut}([Z_1],\cdots,[Z_k])\big|}\cdot \prod_{i=1}^k\dfrac{\big(\op{dim}(Z_i)\big)!}{[N_W(Z_i):W_{Z_i}]}.$$
\end{thm}

In the proof of this theorem, the base space $Y$ is lifted via a ramified covering to a space where the primitive strata intersect transversely to form shadow strata. The result then follows by computing the degrees of these lifted shadow strata and comparing them with the local degrees of the $LL$ map (see \cite[Lemma~5.2.10]{LZ-graphs-on-surfaces}, also \cite[end of p.~15]{LZ-counting-I}). 

The lift of the primitive strata discussed above relies on the interpretation of the $LL$ map in terms of polynomials, their critical points and critical values as discussed in Section~\ref{sec:LL_on_polys};  we have not yet been able to give an analogous construction that works for all reflection groups. Still however, we hope that a similar local analysis of the $LL$ map might be possible in the future, for at least some class of better-behaved passports. The results of \cite{krattenthaler-Muller-decomposition-numbers} along with Prop.~\ref{Prop: local degree of LL and Fact} imply formulas for the local degrees of $LL$ that are, at least in types $B$ and $D$, similar to the previous theorem. We are not aware however, of a general (conjectural) form that would work for all groups and all passports.

\subsection{Hurwitz action on block factorizations}
\label{Section: Hurwitz action on block facto}

As is immediate from its definition (Defn.~\ref{Defn: The (algebraic) Hurwitz action}), the Hurwitz action respects the set of conjugacy classes in a factorization. It is therefore well defined on the set of block factorizations $\sigma$ with a fixed, (unordered) passport $\{\bm Z\}$. The shadow strata $Y_{\{\bm Z\}}$ are again a natural geometric object through which we can study its orbits. The following is a refinement of \cite[Thm.~5.5]{ripoll-french-paper}:

\begin{thm}\label{Thm: Hurwitz action and path connectedness}
The $\op{rlbl}$ map defines a bijection between the path-connected components of the shadow stratum $Y_{\{\bm Z\}}$ and the orbits of the Hurwitz action on block factorizations $\sigma$ with unordered passport $\{\bm Z\}$. In particular, the action is transitive if and only if the stratum is path-connected.
\end{thm} 
\begin{proof}
This is again a corollary of the trivialization theorem (Thm.~\ref{thm: trivialization}). To begin with, notice that after the surjectivity of the $LL\times\op{rlbl}$ map, the theorem is equivalent to checking that two points $y,y'\in Y_{\{\bm Z\}}$ have labels in the same Hurwitz orbit {\it if and only if} they are connected by a path in $Y_{\{\bm Z\}}$. In what follows, let $k$ be the number of the (not necessarily distinct) parabolic classes that constitute the passport $\{\bm Z\}$. 

For the forward direction, pick any two points $y,y'$ in the same component of $Y_{\{\bm Z\}}$ and a path $\beta$ between them. Notice that since $\beta$ stays in $Y_{\{\bm Z\}}$, the multiplicities of the points in $\gamma(t):=LL(\beta(t))\in E_n$ are constant. Therefore, Lemma~\ref{Lem: rlbl is equivariant Hurwitz-path lifting} applies and we have $\op{rlbl}(y')=g*\op{rlbl}(y)$, where $g$ realizes the braid $\gamma(t)$ as an element of the braid group on $k$ strands $B_k$. In other words, the labels of $y$ and $y'$ are in the same Hurwitz orbit.

For the other direction, start with two elements $y,y'\in Y_{\{\bm Z\}}$, and an element $g$ in the braid group $B_k$, such that $\op{rlbl}(y')=g*\op{rlbl}(y)$. Consider now the two configurations $LL(y)$ and $LL(y')$ and a path $\gamma$ between them (in the connected configuration space $E_k\subset E_n$) that gives rise to the braid $g$. After Corol.~\ref{Corol: Path lifting property of LL}, we can lift $\gamma$ to a path $\beta(t):[0,1]\rightarrow Y$, for which $\beta(0)=y$ and $LL(\beta(1))=LL(y')$. Now by Lemma~\ref{Lem: rlbl is equivariant Hurwitz-path lifting}, we will further have that $\op{rlbl}(\beta(1))=g*\op{rlbl}(y)=\op{rlbl}(y')$. Putting these together, we get $$\big( LL\times\op{rlbl}\big) (\beta(1))=\big(LL\times\op{rlbl}\big)(y'),$$ which according to the trivialization theorem forces $\beta(1)=y'$. That is, $y$ and $y'$ are connected by the path $\beta$.
\end{proof}

In certain cases, one can easily describe parametrizations of the shadow strata that imply connectivity. For the symmetric group, this idea was quite successful in \cite{khovanskii-zdravkovska-hurwitz-transitivity}. For well-generated groups, the following theorem by Ripoll tackles primitive factorizations. The proof relies on the geometry of the maps $Z\xrightarrow{\rho}[Z]\xrightarrow{\op{pr}_Y}[Z]_Y$, as in Fig.~\ref{Fig: Lift of the LL map}.

\begin{cor}\cite{ripoll-french-paper}
For an arbitrary flat $Z\in\mathcal{L}_W$, the Hurwitz action is transitive on the set of primitive factorizations of type $[Z]$ (see Defn.~\ref{Defn: parabolics and primitives of type Z}).
\end{cor}

\section*{Acknowledgments}

This paper contains, and is an extension of, results that first appeared in the author's thesis \cite{theo-thesis}. We would like to thank our former advisor Vic Reiner for introducing us to the world of Coxeter combinatorics and guiding us through its many marvelous gems.

We would also like to thank David Bessis, Guillaume Chapuy, Matthieu Josuat-Verg{\'e}s, Joel Lewis, and Vivien Ripoll, for many interesting discussions. Furthermore, we would like to thank the two referees for a very thorough reading of the paper and for their suggestions which strongly improved the presentation of the main results. In particular, Sections 4.3 and 5.4 were added after their recommendation. Finally, we thank Dennis Stanton hoping he will accept the discussion of Section~\ref{Section When N(Z)/W(Z) is a refl group} as a revised answer to one of his---always thought provoking---questions during our defense.

We state that there is no conflict of interest related to this work.

\printbibliography

\Address

\end{document}